\documentclass[11pt,reqno,final]{article}
\usepackage{url}
\usepackage{amsmath,amsfonts,amssymb,bm,amsthm,version}
\usepackage{mathrsfs,fancybox,pifont}
\usepackage{graphicx}
\usepackage{algorithm}
\usepackage{subcaption }
\usepackage{algpseudocode}
\usepackage{booktabs}
\usepackage{cite}
\usepackage{setspace}
\usepackage{url,hyperref}
\usepackage{fullpage} 
\usepackage[notcite,notref]{showkeys}
\usepackage{xcolor}
\usepackage{multirow}
\usepackage{epstopdf}
\usepackage{cases}
\usepackage{geometry}
\usepackage{mathtools}
\usepackage{authblk}
\usepackage{lipsum}
\usepackage{indentfirst}

\newcommand{\dt}{\operatorname{d}\! t}
\newcommand{\ds}{\operatorname{d}\! s}

\newcommand{\dw}{\operatorname{d}\! W}

\allowdisplaybreaks

\numberwithin{equation}{section}

\theoremstyle{plain}
\newtheorem{theorem}{Theorem}
\newtheorem{lemma}{Lemma}
\newtheorem{assumption}{Assumption}

\theoremstyle{definition}

\theoremstyle{remark}
\newtheorem{remark}{Remark}

\title{A Computable Stochastic Riccati Equations Framework for Mean--Variance Portfolio Selection with Multifactor Stochastic Volatility Model}

\author{
Zhecheng Huang, School of Mathematical Sciences Fudan University \& China Pacific Insurance (Group) Co., \href{mailto:huangzhecheng@fudan.edu.cn}{huangzhecheng@fudan.edu.cn}
\and
Guojiang Shao, School of Mathematical Sciences Fudan University, \href{mailto:gjshao23@m.fudan.edu.cn}{gjshao23@m.fudan.edu.cn}
\and   
Lei Wang, China Pacific Insurance (Group) Co.
\and Qi Zhang, School of Mathematical Sciences Fudan University \& State Key Laboratory of Cryptography and Digital Economy Security, Shandong University
}

\begin{document}
\maketitle

\begin{abstract}
We investigate a computable and empirically implementable framework for
continuous-time mean--variance optimal portfolio selection with random market
coefficients. The market model is built on a tractable multifactor stochastic
volatility structure, which captures state-dependent risk premia, stochastic
volatility, and dynamic cross-asset dependence. The optimal control is
characterized by stochastic Riccati equations. On the computational side, we design an iterative BSDE-based procedure to approximate the stochastic Riccati equation from above and below, where the
initial upper and lower bounds are obtained by solving two linear BSDEs. 
We then apply a logarithmic transformation to remove the singularity in the SRE, and solve the
resulting transformed equation using both Deep BSDE and DBDP methods. The linear BSDE bounds also provide effective initial-value estimates, improving the convergence speed and training stability of the Deep BSDE solver. Empirical experiments based on sector ETF data show that the proposed multifactor mean--variance strategy produces smooth target-return wealth dynamics, with improved drawdown control and downside-risk protection relative to benchmark strategies. These results demonstrate the practical potential of combining stochastic Riccati equations, neural BSDE solvers, and multifactor market modeling for dynamic asset allocation.
\end{abstract}


\section{Introduction}

Mean--variance portfolio selection has been a central topic in mathematical finance since the seminal work of Markowitz \cite{markowitz1952modern}. Unlike expected-utility maximization, the mean--variance criterion involves both the expected terminal return and the variance of terminal wealth, leading to a nonseparable objective through the square of an expectation. In continuous time, this difficulty is commonly handled by embedding the original problem into a family of auxiliary stochastic LQ control problems, an approach developed by Zhou and Li \cite{zhou2000continuous}; see also Li and Ng \cite{li2000optimal} for the discrete-time formulation. In non-Markovian markets, this LQ reformulation naturally leads to scalar SREs. The importance of such equations has been well recognized in the stochastic control literature: Kohlmann and Tang \cite{kohlmann2002global} established their global adapted solvability and applied the result to mean--variance hedging, while Lim and Zhou \cite{lim2002mean} and Lim \cite{lim2004quadratic} further developed mean--variance portfolio selection and quadratic hedging in complete and incomplete markets. These works show that scalar SREs are not merely a technical byproduct of the LQ formulation, but a fundamental analytical tool for mean--variance optimization under random market coefficients.

The well-posedness theory of stochastic Riccati equations has been a central issue in stochastic LQ control. The study goes back to the pioneering work of Bismut \cite{bismut1976linear}, where stochastic LQ problems with random coefficients and the associated Riccati-type equations were first investigated. Subsequent developments include the stochastic LQ theory and maximum principle approach summarized in Yong and Zhou \cite{yong1999stochastic}, the indefinite stochastic LQ framework of Chen, Li and Zhou \cite{chen1998stochastic,chen2000stochastic}, and the general random-coefficient theory of Tang \cite{tang2003general}, where the associated stochastic Riccati equation is solved through its connection with linear stochastic Hamiltonian systems.

In the scalar case, stochastic Riccati equations are closely related to BSDEs with quadratic growth. Indeed, the martingale component of a scalar SRE typically enters the generator through quadratic terms, while structural conditions such as the positivity and boundedness of $P$ are essential for the associated feedback representation to be well defined. This connection plays a central role in the one-dimensional theory developed by Kohlmann and Tang \cite{kohlmann2002global} and Lim \cite{lim2004quadratic}. From the BSDE perspective, scalar SREs can be viewed as a distinguished class of quadratic BSDEs. The foundational existence and uniqueness theory for one-dimensional quadratic BSDEs with bounded terminal conditions was established by Kobylanski \cite{kobylanski2000backward}. These results provide the analytical background for treating scalar stochastic Riccati equations through quadratic-BSDE techniques and motivate the logarithmic transformations and monotone iterative schemes considered in this paper.

Classical continuous-time stock models are often formulated under the Black--Scholes framework, where the drift and diffusion coefficients of risky assets are assumed to be constants. Although analytically convenient, this specification is too restrictive to capture important empirical features of financial markets, such as volatility clustering, stochastic fluctuations in risk premia, and volatility smiles or smirks. To overcome these limitations, stochastic factor models have been widely introduced into portfolio selection problems. A representative example is the Heston model \cite{heston1993closed}, in which the variance follows a mean-reverting square-root diffusion. This model captures several facts of asset returns and has been extensively used in option pricing and portfolio optimization. Early contributions to portfolio selection under stochastic volatility include Zariphopoulou \cite{zariphopoulou2001solution}, Kraft \cite{kraft2005optimal}, Liu \cite{liu2001dynamic,liu2007portfolio}, and Chacko and Viceira \cite{chacko2005dynamic}. Related developments further incorporate robust preferences, insurance and reinsurance decisions, and mean--variance criteria under square-root factor models; see, for example, Yi et al. \cite{yi2013robust} and Pham \cite{pham2009continuous}. In particular, Shen \cite{shen2015mean} and Shen and Zeng \cite{shen2015optimal} studied mean--variance problems with unbounded random market coefficients and established the well-posedness of the associated SREs. Their key arguments rely on a change-of-measure method, together with suitable exponential integrability conditions imposed on the unbounded coefficients.

Although the classical Heston model captures stochastic volatility through a single square-root variance factor, empirical studies suggest that volatility dynamics often exhibit richer multifactor structures. Multifactor stochastic volatility models are better suited to describing the persistence, term structure, and cross-sectional behavior of volatility; see, for example, Christoffersen et al. \cite{christoffersen2009shape}. Motivated by this observation, we introduce a multifactor Heston-type market
model for multi-asset portfolio selection. In this model, the drift and
diffusion coefficients are driven by CIR-type latent factors, including
asset-specific volatility factors and a market-wide volatility factor, so as to
capture idiosyncratic fluctuations and common market risk in a unified
framework.

Recent deep learning methods for BSDEs and FBSDEs provide an effective mesh-free framework for solving high-dimensional stochastic control problems and high-dimensional PDEs. The deep BSDE method of Han, Jentzen, and E \cite{han2018solving} approximates the martingale integrand of a BSDE by neural networks along simulated forward paths. For coupled systems, Han and Long \cite{han2020convergence} analyzed the convergence of the deep BSDE method for coupled FBSDEs. Huré et al. \cite{hure2020deep} proposed deep backward dynamic programming (DBDP) schemes, which exploit the BSDE representation and approximate the solution and its gradient recursively by backward induction over the time grid. More closely related to the Hamiltonian structure of stochastic control, Ji et al. \cite{ji2022solving} developed a deep learning approach based on the stochastic maximum principle, in which the original control problem is reformulated as a stochastic Hamiltonian system, or equivalently as an FBSDE with a maximum condition. These works provide the numerical and methodological foundation for our BSDE-based and Hamiltonian-flow-based treatment of stochastic Riccati equations.

The main contribution of this work is threefold. First, we formulate a multi-asset mean--variance investment problem under a multifactor Heston-type market model with both idiosyncratic and market-wide volatility factors, and derive the associated stochastic Riccati equation. Second, we design two deep learning solvers tailored to the Riccati structure: a quadratic BSDE iteration and a logarithmic quadratic-BSDE transformation. Third, we provide a systematic numerical comparison of these algorithms, highlighting their respective advantages in terms of stability, approximation accuracy, and computational efficiency. Our results suggest that incorporating Riccati-specific transformations and stochastic-control structure is crucial for reliable deep learning approximation of mean--variance portfolio problems with random market coefficients.

The closest work to ours is Gnoatto et al.~\cite{gnoatto2025deep}, which develops a deep BSDE method for quadratic hedging in incomplete markets and also involves a stochastic Riccati equation in the mean--variance hedging case. Their Heston-type SRE has a special structure that admits a closed-form representation through deterministic Riccati ODEs. In contrast, we study mean--variance portfolio selection under a multifactor Heston-type market model and focus on the direct numerical treatment of the resulting stochastic Riccati equation. We propose Riccati-oriented BSDE algorithms, including an iterative bounding scheme and a logarithmic transformation method, and test them with both simulated factors and market data. Thus, our focus is on
computational treatment and empirical implementation, while the general
well-posedness problem for SREs with genuinely unbounded random coefficients is
not addressed in this paper.

The rest of the paper is organized as follows. Section~2 presents the
continuous-time mean--variance portfolio problem and the associated stochastic
Riccati equation. Section~3 introduces two BSDE-based algorithms for computing
the SRE. Section~4 develops the multifactor stochastic volatility model and its parameter calibration. Section~5 reports the deep-learning numerical results for the SRE, while Section~6 provides the empirical portfolio analysis and benchmark comparisons. Section~7 concludes. Proofs, implementation details, and additional figures are deferred to the appendices.

\section{Problem Formulation}

Let $(\Omega, \mathcal{F}, \mathbb{P})$ be a complete probability space equipped with an $n$-dimensional standard Brownian motion $\{W(t)\}_{t \in [0,T]}$. Let $\mathbb{F} = \{\mathcal{F}_t\}_{t \in [0,T]}$ denote the natural filtration generated by $W$, augmented by all the $\mathbb{P}$-null sets in $\mathcal{F}$. The superscript $^*$ denotes the transpose of a vector or a matrix, and we say a matrix $A \geq 0$ if $A$ is positive semidefinite almostly surely. We use $I_n$ to denote the $n$-order identity matrix, use $\langle \cdot, \cdot \rangle$ and $|\cdot|$ to denote the standard inner product and the corresponding induced norm in a Euclidean space, respectively. For any $m$, $j \in \mathbb{N}$, we introduce the following spaces that will be used throughout this paper:
\begin{align*}
\begin{array}{rl} 
L_{\mathcal{F}_T}^{\infty}\left( \mathbb{R}^m \right): &\hspace{-0.25cm} \text {the set of $\mathcal{F}_T$-measurable and essentially bounded } \mathbb{R}^m \text{-valued random variables;}\\ 
L_{\mathbb{F}}^{\infty}\left(0, T ; \mathbb{R}^{m \times j} \right): &\hspace{-0.25cm} \text {the set of $\mathbb{F}$-adapted essentially bounded } \mathbb{R}^{m \times j} \text{-valued processes};\\
L_{\mathbb{F}}^{\infty}\left(0, T ; \mathbb{R}_{>0}\right): &\hspace{-0.25cm} \text {the set of $\mathbb{F}$-adapted essentially bounded positive processes};\\
L_{\mathbb{F}}^{\infty}\left(0, T ; \mathbb{R}_{\gg 1}\right): &\hspace{-0.25cm} \text {the set of $\mathbb{F}$-adapted processes } v:[0, T] \times \Omega \rightarrow(0,+\infty) ~\text{such}\\
&\hspace{-0.25cm}\text {that $c^{-1} \leqslant v(t) \leqslant c$ a.e. a.s. for some constant $c>0$};\\
L_{\mathbb{F}}^2\left(0, T ; \mathbb{R}^m \right): &\hspace{-0.25cm} \text {the set of $\mathbb{F}$-adapted processes } v:[0, T] \times \Omega \rightarrow \mathbb{R}^m ~\text {such that } \\
&\hspace{-0.25cm}\mathbb{E}\left[\int_0^T|v(t)|^2 \dt\right]<\infty;\\
L_{\mathbb{F}}^{2,c}\left(0,T; \mathbb{R}\right): &\hspace{-0.25cm} \text {the set of $\mathbb{F}$-adapted processes $v:[0, T] \times \Omega \rightarrow \mathbb{R}$ with} \\
&\hspace{-0.25cm}\text {continuous sample paths such that } \mathbb{E}\left[\sup_{t \in[0, T]}|v(t)|^2\right]<\infty; \\
L_{\mathbb{F}}^{\infty,c}\left(0,T; \mathbb{R}\right): &\hspace{-0.25cm} \text {the set of $\mathbb{F}$-adapted essentially bounded  processes $v:[0, T] \times \Omega \rightarrow \mathbb{R}$ with} \\
&\hspace{-0.25cm}\text {continuous sample paths}. 
\end{array}
\end{align*}

We consider a financial market consisting of $m+1$ securities: one risk-free bond and $m$ risky stocks. The bond price process, denoted by $S_0(t)$, evolves according to the following ordinary differential equation (ODE):
$$
\left\{\begin{array}{l}
\mathrm{d} S_0(t) = r(t) S_0(t) \dt, \qquad t \in (0, T], \\
S_0(0) = s_0 > 0,
\end{array}\right.
$$
where the interest rate $r(\cdot) \in L_{\mathbb{F}}^{\infty}(0, T ; \mathbb{R}_{>0})$. The price process of the $i$-th stock, $S_i(t)$ for $i=1, \cdots, m$, is governed by the following stochastic differential equation (SDE):
$$
\left\{\begin{array}{l}
\mathrm{d} S_i(t) = S_i(t)\left\{\mu_i(t) \dt + \sum_{j=1}^n \sigma_{i j}(t) \dw_j(t)\right\}, \qquad t \in(0, T], \\
S_i(0) = s_i > 0,
\end{array}\right.
$$
where $\mu_i(t)$ and $\sigma_i(t) := [\sigma_{i 1}(t), \cdots, \sigma_{i n}(t)]$ represent the appreciation rate and the volatility (or dispersion) vector of the $i$-th stock, respectively. We further assume that $\mu_i(\cdot) \in L_{\mathbb{F}}^{\infty}(0, T ; \mathbb{R}_{>0})$ and $\sigma_{i j}(\cdot) \in L_{\mathbb{F}}^{\infty}(0, T ; \mathbb{R}_{\gg 1})$.

\begin{remark}
    Note that the assumption $r(t)>0$ is imposed merely to reflect typical financial market conditions. The main results of this paper, however, do not rely on this strict positivity, the same holds true for the assumptions placed on $\mu_i(t)$ and $\sigma_{ij}(t)$.
\end{remark}
Define the appreciation rate vector $\mu(t)$ and the volatility matrix $\sigma(t)$ as follows:
$$
\mu(t) := \left[\begin{array}{c} \mu_1(t) \\ \vdots \\ \mu_m(t) \end{array}\right] \in \mathbb{R}^{m} \qquad \text{and} \qquad \sigma(t) := \left[\begin{array}{c} \sigma_1(t) \\ \vdots \\ \sigma_m(t) \end{array}\right] \in \mathbb{R}^{m \times n}.
$$
Consequently, we have $\mu(\cdot) \in L_{\mathbb{F}}^{\infty}(0, T ; \mathbb{R}^{m})$ and $\sigma(\cdot) \in L_{\mathbb{F}}^{\infty}(0, T ; \mathbb{R}^{m \times n})$. We assume throughout this paper that $\sigma(t)$ is uniformly non-degenerate; that is, there exists a constant $\varepsilon > 0$ such that
$$
\sigma(t) \sigma^*(t) \geqslant \varepsilon I_{m \times m}, \quad \text{for all } t \in [0, T], \quad \mathbb{P}\text{-a.s.}
$$
Let $\mathbf{1}_m = [1,1, \cdots, 1]^* \in \mathbb{R}^m$, and define the excess return vector as $\tilde{\mu}(t) := \mu(t) - r(t)\mathbf{1}_m \in \mathbb{R}^m$. The risk premium process $\rho(t)$ is then given by
$$
\rho(t) := \sigma^*(t)(\sigma(t) \sigma^*(t))^{-1} \tilde{\mu}(t) \in \mathbb{R}^{n}, \qquad t \in [0, T].
$$

Let $X(t)$ denote the total wealth of an agent at time $t \in [0, T]$. Assuming continuous-time trading with no transaction costs or intermediate consumption, the wealth process evolves according to the following SDE:
\begin{equation} \label{state}
\left\{\begin{aligned}
\mathrm{d}X(t) &= \left\{r(t) X(t) + \sum_{i=1}^m [\mu_i(t) - r(t)] \pi_i(t)\right\} \dt + \sum_{i=1}^m \sum_{j=1}^n \sigma_{i j}(t) \pi_i(t) \dw_j(t), \quad t \in (0, T], \\
X(0) &= x_0 > 0,
\end{aligned}\right.
\end{equation}
where $\pi_i(t)$ represents the total market value invested in the $i$-th risky asset at time $t$. A negative value $\pi_i(t) < 0$ for $i \in \{1, \cdots, m\}$ indicates short-selling the $i$-th stock, while $\pi_0(t) < 0$ implies borrowing at the risk-free rate $r(t)$. We define the agent's portfolio strategy as the vector process $\pi(t) := [\pi_1(t), \cdots, \pi_m(t)]^*$. Note that the amount invested in the bond, $\pi_0(t)$, is excluded from $\pi(t)$ because it is entirely determined by the self-financing condition $\pi_0(t) = X(t) - \sum_{i=1}^m \pi_i(t)$. To proceed, we denote rewrite the state equation \eqref{state} in the following compact form:
\begin{equation} \label{dynamics}
    \left\{\begin{aligned}
        \mathrm{d}X(t) &= [r(t)X(t) + \langle \tilde{\mu}(t), \pi(t) \rangle]\dt + \langle \sigma^*(t)\pi(t), \dw(t) \rangle, \\
        X(0) &= x_0.
    \end{aligned}\right.
\end{equation}

The class of admissible controls is defined as the set
$$
\mathcal{A}:=\left\{\pi(\cdot) \in L_{\mathcal{F}}^2\left(0, T ; \mathbb{R}^m\right) \right\} .
$$

If $u(\cdot) \in \mathcal{A}$ and $X(\cdot)$ is the associated solution of \eqref{dynamics}, then we refer to $(X(\cdot), \pi(\cdot))$ as an admissible pair.
The agent's objective is to find an admissible portfolio $\pi(\cdot) \in \mathcal{A}$ that minimizes the risk, measured by the variance of the terminal wealth, while achieving a prescribed expected terminal wealth $\mathbb{E}[X(T)] = d$ for a given target $d \in \mathbb{R}$. Note that under this expectation constraint, the variance can be expressed as
$$
\operatorname{Var}(X(T)) = \mathbb{E}\left[|X(T) - \mathbb{E}[X(T)]|^2\right] = \mathbb{E}\left[|X(T) - d|^2\right].
$$

and formulate the following constrained stochastic control problem, parameterized by the target $d \in \mathbb{R}$:
\begin{equation} \label{prob1}
    \begin{array}{rl}
        \min\limits_{\pi(\cdot) \in \mathcal{A}} & \mathcal{J}(\pi(\cdot), d) \triangleq \operatorname{Var}(X(T)) = \mathbb{E}\left[|X(T) - d|^2\right], \\
        \text{subject to} & \left\{
        \begin{array}{l}
            \mathbb{E}[X(T)] = d, \\
            (X(\cdot), \pi(\cdot)) \text{ satisfies } \eqref{dynamics}.
        \end{array}
        \right.
    \end{array}
\end{equation}

{ \bf Problem \ref{prob1} } is a classic stochastic control problem, whose solution procedure and feasibility condition can be found in many references (e.g. kohlmann and Tang \cite{kohlmann2002global}, Lim \cite{lim2004quadratic}). We will provide a brief solution procedure in the Appendix \ref{solving-MV}. In particular, the optimal strategy is associated with the following SRE with random coefficients:
\begin{equation} \label{sre}
    \left\{\begin{aligned}
        \mathrm{d}P &= - \left[(2r-|\rho|^2)P - 2 \langle \rho,\Lambda \rangle - \frac{1}{P} \Lambda^*\sigma^* (\sigma\sigma^*)^{-1} \sigma \Lambda \right]\dt + \langle \Lambda,\dw\rangle, \quad t \in [0, T),  \\
        P(T) & = 1, \\
        P(t) & > 0, \qquad t \in [0, T), \\
    \end{aligned} \right.
\end{equation}
the numerical schemes for solving this type SRE \eqref{sre} are the core contribution of this paper whose specific details will be provided in the next section. Furthermore, for notational simplicity, we will frequently suppress the time variable $t$ in processes and parameters when no confusion arises.

A major technical challenge lies in computing the value of $(P,\Lambda)$ and optimal stragety $\pi^*$ given by \eqref{strategy}, which requires solving the highly singular stochastic Riccati equation \eqref{sre}. To bridge the gap between theoretical existence and numerical tractability, we propose different numerical methods below. For later use, we recall the standard results for BSDEs with quadratic growth in $z$ in Appendix \ref{appendix-quadratic}.

\section{Numerical Schemes for SRE}

\subsection{Numerical Scheme Through Iterative BSDEs with Quadratic Growth}

Below, we will provide the monotone iterative scheme of SRE \eqref{sre}, define the approximating sequence $\{P_k, \Lambda_k\}_{k \geq 1}$ iteratively. Let $(P_1, \Lambda_1)$ be the solution to the BSDE:
\begin{equation} \label{sre1}
    \left\{\begin{aligned}
        \mathrm{d}P_1 &= - \left[(2r-|\rho|^2)P_1 - 2 \langle \rho,\Lambda_1 \rangle - \frac{1}{P_1} \Lambda_1^* \Lambda_1 \right]\mathrm{d}t + \langle \Lambda_1,\mathrm{d}W \rangle, \qquad t \in [0,T), \\
        P_1(T) &= 1.
    \end{aligned} \right.
\end{equation}
The well-posedness of BSDE \eqref{sre1} is a classical result, we note that $\tilde{P}_1 := 1/P_1$ and $\tilde{\Lambda}_1 := -\Lambda_1/P_1^2$ satisfies a linear BSDE
\begin{equation} \label{linear-P1}
    \mathrm{d}\tilde{P}_1 = [(2r-|\rho|^2)\tilde{P}_1 + 2 \langle \rho,\tilde{\Lambda}_1 \rangle ]\mathrm{d}t + \langle \tilde{\Lambda}_1,\mathrm{d}W\rangle,
\end{equation}
which guarantees $\tilde{P}_1 \in L_{\mathbb{F}}^{\infty,c}(0, T; \mathbb{R}_{\gg 1})$ and thus $P_1 \in L_{\mathbb{F}}^{\infty,c}(0, T; \mathbb{R}_{\gg 1})$ is well-defined. For $k \geq 1$, we construct the iteration:
\begin{equation} \label{iteration_P}
    \left\{\begin{aligned}
        \mathrm{d}P_{k+1} &= - \left[(2r-|\rho|^2)P_{k+1} - 2 \langle \rho,\Lambda_{k+1} \rangle - \frac{1}{P_k} \Lambda_{k+1}^* \Pi \Lambda_{k+1} \right]\mathrm{d}t + \langle \Lambda_{k+1},\mathrm{d}W \rangle, \quad t \in [0,T), \\
        P_{k+1}(T) &= 1,
    \end{aligned} \right.
\end{equation}
where $\Pi := \sigma^* (\sigma\sigma^*)^{-1} \sigma$ is the orthogonal projection matrix. The well-posedness of BSDE \eqref{iteration_P} is immediately derived from Theorem \ref{BSDE-quadratic} by induction.

Symmetrically, we define the reciprocal system to obtain an upper bound. Let $K := 1/P$ and $\Xi := -\Lambda/P^2$. Applying It\^{o}'s formula, $K$ satisfies a BSDE where the singular term involves the matrix $I_{n} - \Pi$:
\begin{equation*}
    \left\{\begin{aligned}
        \mathrm{d} K &= \left[\left(2 r-|\rho|^2\right) K + 2 \langle\rho, \Xi\rangle + \frac{1}{K} \Xi^*\left(I_n - \Pi \right) \Xi\right] \dt+\langle\Xi, \dw \rangle, \qquad t \in [0, T),  \\
        K(T) &= 1.
    \end{aligned}\right.
\end{equation*}
We then define the iterative sequence $\{K_k, \Xi_k\}_{k \geq 1}$,
\begin{equation} \label{sre2}
    \left\{\begin{aligned}
    \mathrm{d} K_1 &= \left[\left(2 r-|\rho|^2\right) K_1 + 2\langle\rho, \Xi_1\rangle + \frac{1}{K_1} \Xi_1^* \Xi_1 \right] \dt+\langle\Xi_1, \dw \rangle, \qquad t \in [0, T),  \\
    K_1(T) &= 1,
    \end{aligned}\right.
\end{equation}
and
\begin{equation} \label{iterative_K}
    \left\{\begin{aligned}
    \mathrm{d} K_{k+1} & = \left[\left(2 r-|\rho|^2\right) K_{k+1} + 2\langle\rho, \Xi_{k+1} \rangle+ \frac{1}{K_k} \Xi_{k+1}^*(I_{n} - \Pi) \Xi_{k+1} \right] \mathrm{d}t+\langle \Xi_{k+1}, \mathrm{d}W \rangle,  \quad t \in [0, T),  \\
    K_{k+1}(T) &= 1,
    \end{aligned}\right.
\end{equation}
for $k \geq 1$. Similarly, if we assume that $\tilde{K}_1 := 1 / K_1$ and $\tilde{\Xi}_1 := - \Xi_1 / K_1^2$, then $(\tilde{K}_1 ,\tilde{\Xi}_1 )$ satisfies
\begin{equation} \label{linear-K1}
    \mathrm{d}\tilde{K}_1 =  [-(2r-|\rho|^2)\tilde{K}_1 + 2 \langle \rho,\tilde{\Xi}_1 \rangle ]\mathrm{d}t + \langle \tilde{\Xi}_1,\mathrm{d}W\rangle,
\end{equation}
which implies the well-posedness of BSDE \eqref{sre2}, the well-posedness of BSDE \eqref{iterative_K} can be proved similar to \eqref{iteration_P}. Noting that $z^* (I_{n} - \Pi) z = \inf_{x \in \mathbb{R}^m} |z - \sigma^* x|^2 \geq 0$, the comparison principle in Theorem \ref{BSDE-quadratic} implies the monotonicity of these sequences. More specifically, we have the following inequality estimates:

\begin{theorem} \label{bounds}
    The sequences $\{P_k\}$ and $\{K_k\}$ provide monotonically increasing and decreasing bounds for the solution $P$ of the original SRE \eqref{sre}:
    \begin{equation} \label{algo-one}
        P_1 \leq P_2 \leq \cdots \leq P_k  \leq \cdots \leq P = \frac{1}{K} \leq \cdots \leq \frac{1}{K_k} \leq \cdots \leq \frac{1}{K_1}, \quad a.s.
    \end{equation}
    Furthermore, as $k \to \infty$, we have the following convergences in $L_{\mathbb{F}}^2(0,T;\mathbb{R})$:
    \begin{equation*}
        \lim_{k \to \infty} \mathbb{E} \int_{0}^{T} \left\{ |P_k(t) - P(t)|^2 + |\Lambda_k(t) - \Lambda(t)|^2 \right\} \mathrm{d}t = 0,
    \end{equation*}
    and
    \begin{equation*}
        \lim_{k \to \infty} \mathbb{E} \int_{0}^{T} \left\{ |P(t)K_k(t) - 1|^2 + |P^2(t)\Xi_k(t) + \Lambda(t)|^2 \right\} \mathrm{d}t = 0.
     \end{equation*}
\end{theorem}

\begin{proof}
    We first establish the monotonicity of $\{P_k\}_{k \ge 1}$ by induction. For $k=1$, noting that $\Pi \geq 0$ and $I_{n} - \Pi \geq 0$, the comparison principle (Theorem \ref{BSDE-quadratic}) is applied to $P_1$ and $P_2$ by the relation between their generators. For the inductive step, we define the generator of \eqref{iteration_P} as $g(P_{k+1}, \Lambda_{k+1}; P_k)$. Since $g$ is non-decreasing with respect to the parameter $P_k$, the ordering $P_k \le P_{k+1}$ implies $P_{k+1} \le P_{k+2}$ via the comparison theorem. The uniform boundedness of $\{P_k\}$ in $L^\infty_{\mathbb{F}}(0,T;\mathbb{R})$ is then derived from standard a priori estimates for BSDEs with quadratic growth.

    Regarding convergence, since the sequence $\{P_k\}$ is monotonic and bounded, it converges point-wise $\dt \times \mathrm{d}\mathbb{P}$-a.s. to a limit process $P$. Following Proposition $5.4$ in Lim \cite{lim2004quadratic}, the dominated convergence theorem and the stability of BSDEs ensure that $P_k \to P$ strongly in $L^2_{\mathbb{F}}(0,T;\mathbb{R})$. The convergence of the martingale components $\Lambda_k \to \Lambda$ in $L^2_{\mathbb{F}}(0,T;\mathbb{R}^n)$ is a direct consequence of applying It\^{o}'s formula to $|P_k - P|^2$ and utilizing the energy inequalities as shown in Lemma $A.2$ in Lim \cite{lim2004quadratic}.

    A symmetric analysis applies to the reciprocal sequence $\{(K_k, \Xi_k)\}_{k \ge 1}$. Noting that the matrix $I_{n} - \Pi$ is positive semi-definite, the comparison principle ensures the monotonicity of $\{K_k\}$. By similar arguments, we have the strong convergence in $L^2_{\mathbb{F}}(0,T;\mathbb{R}^n)$:
\begin{equation*}
    \mathbb{E} \int_{0}^{T} \left\{ |K_k(t) - K(t)|^2 + |\Xi_k(t) - \Xi(t)|^2 \right\} \mathrm{d}t \to 0 \quad \text{as } k \to \infty,
\end{equation*}
where $(K, \Xi)$ is the solution to the SRE for the reciprocal process. Finally, by substituting the algebraic relationships $K = 1/P$ and $\Xi = -\Lambda/P^2$, and utilizing the uniform boundedness of $P(\cdot)$, we obtain the convergence in terms of the original variables:
\begin{equation*}
    \mathbb{E} \int_{0}^{T} \left\{ |P(t)K_k(t) - 1|^2 + |P^2(t)\Xi_k(t) + \Lambda(t)|^2 \right\} \mathrm{d}t \to 0 \quad \text{as } k \to \infty.
\end{equation*}
\end{proof}

By treating $P_k$ (or $K_k$) from the previous iteration as a known parameter in the generator of the $(k+1)$-th step, the problem is reduced to solving a sequence of quadratic BSDEs with simpler structures. In fact, computing \((P_1,\Lambda_1)\) and \((K_1,\Xi_1)\) can be reduced to computing the simpler linear BSDEs \((\tilde{P}_1,\tilde{\Lambda}_1)\) and \((\tilde{K}_1,\tilde{\Xi}_1)\), which makes it convenient for us to numerically estimate the upper and lower bounds of $P$. Moreover, in a deterministic market, the martingale components vanish, and these bounds become tight with $P_1(t) = P_k(t) = P(t) = 1/K_k(t) = 1/K_1(t)$, recovering the classical Riccati ODE solution.

\begin{remark}
    Actually, if the market is complete, thus $\sigma$ is invertible (rank $ \sigma = n$), it's easy to check BSDE \eqref{sre1} is the same as SRE \eqref{sre}, if $\sigma = 0$ (rank $\sigma=0$), we can find $(P,\Lambda)=(\tilde{K}_1,\tilde{\Xi}_1)$. That is to say, the two ends of the inequality column \eqref{algo-one} depict two extreme cases of the market.
\end{remark}

The monotone iterative structure presented in \eqref{iteration_P} and \eqref{iterative_K} provides a numerical method to calculate the solution of SRE \eqref{sre}, however, since the approximation through $P_k$ and $K_k$ is based on the monotonic convergence theorem rather than the contraction mapping, it is difficult to provide an estimate of the convergence speed for \eqref{algo-one}.

\subsection{Numerical Scheme Through Logarithmic Transformation}

Next we provide a more direct method, we apply the logarithmic transformation to SRE \eqref{sre}, and define
\begin{equation*}
    \check{P}(t):=\ln P(t), \qquad \check{\Lambda}(t):=\frac{\Lambda(t)}{P(t)},
\end{equation*}
using It\^o's formula to $\check{P}(t)$
\begin{equation} \label{log-transformation}
    \left\{\begin{aligned}
        \mathrm{d}\check{P} &= - \left[(2r-|\rho|^2) - 2 \langle \rho, \check{\Lambda} \rangle -  \check{\Lambda}^*\sigma^* (\sigma\sigma^*)^{-1} \sigma  \check{\Lambda} + \frac{1}{2} | \check{\Lambda}|^2 \right]\dt + \langle  \check{\Lambda},\dw\rangle, \quad t \in [0, T),  \\
        \check{P}(T) & = 0,
    \end{aligned} \right.
\end{equation}
hence, the original SRE \eqref{sre} is transformed into a quadratic BSDE \eqref{log-transformation}. Therefore, by Theorem \ref{BSDE-quadratic}, the transformed BSDE \eqref{log-transformation} admits a unique adapted solution \((\check{P},\check{\Lambda})\). Consequently,
\[
    P=e^{\check{P}}, \qquad \Lambda=P\check{\Lambda},
\]
gives the unique strictly positive solution to the original SRE \eqref{sre}.

Moreover, after the logarithmic transformation, the transformed BSDE \eqref{log-transformation} no longer contains singular coefficients near $0$. This reformulation is particularly advantageous for numerical solution, since it avoids potential instabilities caused by divisions by small values of \(P(t)\), and leads to a more stable quadratic BSDE structure for both theoretical analysis and numerical schemes.

\begin{remark}
    The ordering relation \eqref{algo-one} is specific to the scalar SRE considered here. Indeed, its proof relies essentially on the comparison principle for scalar quadratic BSDEs, which does not generally extend to multidimensional quadratic BSDEs. Similarly, the logarithmic transformation applied to \(P\) is valid in this scalar setting, but has no direct analogue for a genuinely multidimensional or matrix-valued SRE.
\end{remark}

\subsection{Comparison of the Two Schemes}

The two schemes introduced above play complementary roles in the deep BSDE implementation. In the deep BSDE method, the initial value of the BSDE is treated as a trainable parameter, and an inaccurate initialization may lead to slow convergence or poor training performance. For the iterative scheme, the first-step equations can be reduced to linear BSDEs \eqref{linear-P1} and \eqref{linear-K1}, whose initial values can be numerically evaluated from their explicit representations. These values therefore provide reliable initial estimates for the subsequent quadratic BSDEs.

In particular, their values at $t=0$ yield a reasonable interval for the unknown initial value $P(0)$, and hence for the transformed initial value. In fact, we can take
\begin{equation*}
    \log \frac{1}{2} \left( \frac{1}{\tilde{P}_1} + \tilde{K}_1 \right),
\end{equation*}
as the initial value of \eqref{log-transformation}. Once a sufficiently accurate initial estimate is available, however, directly solving the logarithmically transformed SRE \eqref{log-transformation} is generally more convenient and computationally efficient than repeatedly solving the iterative BSDEs. Therefore, in the numerical experiments below, we mainly employ the logarithmic-transformation approach within the deep BSDE framework, while using the iterative scheme to provide initial-value estimates and reference bounds.

\section{Model Specifications and Parameter Calibration}

\subsection{The Markovian Market Conditions}

We first impose a Markovian factor structure on the market coefficients, which will be used throughout the subsequent implementation. Specifically, we assume that the market randomness is driven by a Markovian factor process \((\bm V_s)_{s\in[0,T]}\) satisfying
\begin{equation} \label{factor-process}
    \mathrm{d}\bm V(s) = \bm F(\bm V(s))\,\mathrm{d}s + \bm G(\bm V(s))\,\mathrm{d}W(s), \qquad s\in[0,T],
\end{equation}
where
\[
    \bm F:\mathbb R^n\to\mathbb R^n, \qquad \bm G:\mathbb R^n\to\mathbb R^{n\times n},
\]
are deterministic functions. The market coefficients are assumed to admit the
following Markovian structure:
\begin{equation*}
    r(s,\omega)=r(s,\bm V(s)), \qquad
    \mu(s,\omega)=\mu(s,\bm V(s)), \qquad
    \sigma(s,\omega)=\sigma(s,\bm V(s)),
    \qquad s\in[0,T].
\end{equation*}
Here
\[
    r:[0,T]\times\mathbb R^n\to\mathbb R, \qquad
    \mu:[0,T]\times\mathbb R^n\to\mathbb R^m, \qquad
    \sigma:[0,T]\times\mathbb R^n\to\mathbb R^{m\times n},
\]
are bounded continuous functions. Moreover, \(\sigma(s,v)\sigma(s,v)^*\) is
uniformly elliptic on the relevant domain; namely, there exists a constant
\(\delta>0\) such that
\begin{equation*}
    z^*\sigma(s,v)\sigma(s,v)^*z
    \geq
    \delta |z|^2,
    \qquad
    \forall z\in\mathbb R^m,
\end{equation*}
for all \((s,v)\) in the domain under consideration.

\subsection{A Multi-factor Stochastic Volatility Model} \label{empirical}

In this subsection, we present the detailed specification of the proposed multifactor stochastic volatility model and describe how its parameters are calibrated using market data. Throughout this subsection, unless otherwise stated, all Brownian motions are assumed to be one-dimensional standard Brownian motions defined on a complete filtered probability space. All model coefficients are taken to be deterministic constants, and all Brownian motions are mutually independent unless explicitly stated otherwise.

Before introducing our multifactor stochastic volatility framework, we briefly review the classical one-factor Heston model \cite{heston1993closed}, which is one of the most widely used models in the option pricing literature. The stock price process satisfies
\begin{equation*}
    \mathrm{d}S = rS\,\mathrm{d}t + \sqrt{V}\,S\,\mathrm{d}z_1,
\end{equation*}
where the variance process follows a square-root diffusion
\begin{equation*}
    \mathrm{d}V = (a - bV)\,\mathrm{d}t + \sigma \sqrt{V}\,\mathrm{d}z_2.
\end{equation*}
The Brownian motions $z_1$ and $z_2$ are correlated with coefficient $\rho_0$. This mean-reverting square-root specification captures key empirical features such as stochastic volatility and the leverage effect. To better describe the term structure of volatility, Christoffersen et al.~\cite{christoffersen2009shape} proposed a two-factor extension:
\begin{equation*}
    \mathrm{d}S = rS\,\mathrm{d}t + \sqrt{V_1}S\,\mathrm{d}z_3 + \sqrt{V_2}S\,\mathrm{d}z_4,
\end{equation*}
where $V_1$ and $V_2$ follow independent CIR-type processes
\begin{equation*}
    \begin{aligned}
\mathrm{d}V_1 &= (a_1 - b_1 V_1)\,\mathrm{d}t + \sigma_1 \sqrt{V_1}\,\mathrm{d}z_5, \\
\mathrm{d}V_2 &= (a_2 - b_2 V_2)\,\mathrm{d}t + \sigma_2 \sqrt{V_2}\,\mathrm{d}z_6.
\end{aligned}
\end{equation*}
Leverage effects are incorporated by allowing $z_3$ and $z_5$ (resp. $z_4$ and $z_6$) to be correlated, while all other Brownian motions are independent.

Although this model captures richer volatility structures, it remains restricted to a single asset and does not explicitly model cross-asset dependence. Motivated by these limitations, we consider a financial market consisting of four risky assets and one risk-free asset, where both idiosyncratic and systematic sources of uncertainty are present. Define the variance factor vector
\begin{equation*}
    \mathbf{V}(t) = (V_0(t), V_1(t), V_2(t), V_3(t), V_4(t))^*,
\end{equation*}
where each component follows a square-root process
\begin{equation} \label{CIR}
    \mathrm{d}V_k = (\alpha_k - \beta_k V_k)\,\mathrm{d}t + \sigma_k \sqrt{V_k}\,\mathrm{d}Z_k, 
    \qquad k=0,1,2,3,4.
\end{equation}
Here $V_0$ represents the stochastic volatility of the overall market, while $V_k$ ($k=1,\cdots,4$) describe asset-specific volatility factors.

The market index $P_0$ evolves according to
\begin{equation*}
    \frac{\mathrm{d}P_0}{P_0} = \mu_0(\mathbf{V})\,\mathrm{d}t + \sqrt{V_0}\,\mathrm{d}B_{0,0},
\end{equation*}
where
\begin{equation*}
    \mathrm{d}B_{0,0} = \nu_0\,\mathrm{d}Z_0 + \sqrt{1-\nu_0^2}\,\mathrm{d}W_0.
\end{equation*}
Here $\nu_0$ (and similarly $\nu_k$ in the sequel) is assumed to be a suitably specified linear function of $\mathbf{V}$, its explicit form will be provided later. This specification preserves the classical Heston-type structure at the aggregate market level. For each risky asset $P_k$, $k=1,\cdots,4$, we assume
\begin{equation*}
    \frac{\mathrm{d}P_k}{P_k}
    =
    \mu_k(\mathbf{V})\,\mathrm{d}t
    + \sqrt{V_k}\,\mathrm{d}B_k
    + \delta_{k,0}\sqrt{V_0}\,\mathrm{d}W_{k,0}
    + \gamma_{k,0}\sqrt{V_0}\,\mathrm{d}B_{k,0}.
\end{equation*}

The driving Brownian motions admit the decompositions
\begin{equation*}
    \mathrm{d}B_k = \nu_k\,\mathrm{d}Z_k + \sqrt{1-\nu_k^2}\,\mathrm{d}W_k,
\end{equation*}
\begin{equation*}
    \mathrm{d}W_{k,0} = \alpha_{k,0}\,\mathrm{d}W_0 + \sqrt{1-\alpha_{k,0}^2}\,\mathrm{d}Y_{k,0},
\end{equation*}
\begin{equation*}
    \mathrm{d}B_{k,0} = \rho_{k,0}\,\mathrm{d}Z_0 + \sqrt{1-\rho_{k,0}^2}\,\mathrm{d}Z_{k,0}.
\end{equation*}

The risk-free asset satisfies
\begin{equation*}
    \mathrm{d}S_0 = r S_0\,\mathrm{d}t.
\end{equation*}

Following Liu \cite{liu2007portfolio}, we specify the drift as
\begin{equation*}
    \mu_k(\mathbf{V}) = r + m_k V_k + n_k V_0.
\end{equation*}

The above dynamics admit a natural economic interpretation. The term $\sqrt{V_k}\,\mathrm{d}B_k$ represents idiosyncratic risk, driven by asset-specific stochastic volatility. The correlation between $B_k$ and $Z_k$ introduces a leverage effect at the individual asset level. 

The terms $\delta_{k,0}\sqrt{V_0}\,\mathrm{d}W_{k,0}$ and $\gamma_{k,0}\sqrt{V_0}\,\mathrm{d}B_{k,0}$ capture the exposure to systematic market risk. The former corresponds to market shocks that are orthogonal to volatility innovations, such as macroeconomic or liquidity shocks, while the latter reflects the impact of volatility shocks from the aggregate market, thereby introducing a systematic leverage effect. Through this structure, cross-asset dependence is generated endogenously via the common factor $V_0$ and shared Brownian drivers, leading to stochastic and state-dependent correlations.

Compared with the classical Heston model \cite{heston1993closed}, the proposed framework extends stochastic volatility to a multi-asset setting and incorporates both idiosyncratic and common volatility factors. Compared with multi-factor stochastic volatility models such as Christoffersen et al. \cite{christoffersen2009shape}, our model explicitly captures cross-asset interactions through a common market factor rather than modeling each asset independently. In contrast to traditional multi-asset diffusion models with constant covariance structures, the proposed model produces stochastic correlations driven by underlying variance factors, which is more consistent with empirical evidence.

Overall, the model provides a unified and tractable framework that simultaneously captures stochastic volatility, leverage effects, and dynamic cross-asset dependence. Moreover, following Liu \cite{liu2007portfolio}, by linking expected returns to both idiosyncratic and systematic volatility through the specification $\mu_k(\mathbf{V}) = r + m_k V_k + n_k V_0$, here the parameters $m_k$ and $n_k$ are to be estimated from empirical data using appropriate statistical methods, such as maximum likelihood estimation.

Notice that the Brownian motions $Y_{k,0}$ and $Z_{k,0}$ are independent. Then we obtain
\begin{equation*}
\begin{aligned}
    & \delta_{k,0}\sqrt{V_0}\,\mathrm{d}W_{k,0} + \gamma_{k,0}\sqrt{V_0}\,\mathrm{d}B_{k,0} \\
    =\;& \delta_{k,0}\sqrt{V_0}\left(\alpha_{k,0}\,\mathrm{d}W_0 + \sqrt{1-\alpha_{k,0}^2}\,\mathrm{d}Y_{k,0}\right) 
    + \gamma_{k,0}\sqrt{V_0}\left(\rho_{k,0}\,\mathrm{d}Z_0 + \sqrt{1-\rho_{k,0}^2}\,\mathrm{d}Z_{k,0}\right) \\
    =\;& \alpha_{k,0} \delta_{k,0} \sqrt{V_0}\,\mathrm{d}W_0 
    + \sqrt{V_0} \left( \gamma_{k,0} \rho_{k,0}\,\mathrm{d}Z_0 
    + \delta_{k,0}\sqrt{1-\alpha_{k,0}^2}\,\mathrm{d}Y_{k,0} 
    + \gamma_{k,0}\sqrt{1-\rho_{k,0}^2}\,\mathrm{d}Z_{k,0} \right) \\
    =\;& \alpha_{k,0} \delta_{k,0} \sqrt{V_0}\,\mathrm{d}W_0 
    + \sqrt{V_0} \left( \gamma_{k,0} \rho_{k,0}\,\mathrm{d}Z_0 
    + \sqrt{\delta_{k,0}^2(1-\alpha_{k,0}^2) + \gamma_{k,0}^2(1-\rho_{k,0}^2)}\,\mathrm{d}Z'_{k,0} \right),
\end{aligned}
\end{equation*}
where $Z'_{k,0}$ is a standard Brownian motion obtained via orthogonalization. 
Therefore, without loss of generality, we may set $\alpha_{k,0}=1$ by appropriately choosing the parameters $\delta_{k,0}$, $\gamma_{k,0}$, and $\rho_{k,0}$ in what follows.

Based on the above market model with one risk-free asset and four risky assets, we formulate the wealth dynamics of a self-financing portfolio in the standard form \eqref{dynamics} where $\pi(t) = (\pi_1(t),\pi_2(t),\pi_3(t),\pi_4(t))^* \in \mathbb{R}^4$ denotes the amount invested in the four risky assets. And we assume the excess return vector $\tilde{\mu}(t)$ is given by
\begin{equation*}
    \tilde{\mu}(t) 
    = \big(m_1 V_1(t) + n_1 V_0(t),\,\cdots,\,m_4 V_4(t) + n_4 V_0(t) \big)^*,
\end{equation*}
which captures both asset-specific and market-wide stochastic risk premia. To represent the diffusion term, we collect all independent Brownian motions into the vector
\begin{equation*}
    W(t) = \big(Z_1(t),\cdots,Z_4(t),\, W_1(t),\cdots,W_4(t),\, W_0(t),\, Z_{1,0}(t),\cdots,Z_{4,0}(t),\, Z_0(t)\big)^*.
\end{equation*}

The volatility matrix $\sigma(t) \in \mathbb{R}^{4 \times d}$ (with $d=14$) is constructed according to the decomposition of each asset return. Specifically, for each asset $k=1,2,3,4$, its return dynamics admit the representation
\begin{equation*}
\begin{aligned}
    \frac{\mathrm{d}S_k}{S_k}
    =&\ \mu_k(\mathbf{V})\,\mathrm{d}t 
    + \sqrt{V_k}\Big( \nu_k \mathrm{d}Z_k + \sqrt{1-\nu_k^2}\mathrm{d}W_k \Big) 
    + \delta_{k,0}\sqrt{V_0} \mathrm{d}W_0  \\
    &+ \gamma_{k,0}\sqrt{V_0}\Big( \rho_{k,0}\mathrm{d}Z_0 + \sqrt{1-\rho_{k,0}^2}\mathrm{d}Z_{k,0} \Big),
\end{aligned}
\end{equation*}
where the first term represents idiosyncratic volatility, while the latter two terms capture exposure to common market shocks. The drift term $\mu_k(\mathbf{V})$ is given by $m_k V_k (t)+ n_k V_0(t) + r$ according to the definition of excess return vector $\tilde{\mu}(t)$.

Accordingly, the $k$-th row of the volatility matrix $\sigma(t)$ is given by the coefficients associated with the corresponding Brownian motions:
\begin{equation*}
\begin{aligned}
    &\sigma_{k,k} = \sqrt{V_k}\nu_k, \qquad 
    \sigma_{k,k+4} = \sqrt{V_k}\sqrt{1-\nu_k^2}, \\
    &\sigma_{k,9} = \delta_{k,0}\sqrt{V_0}, \qquad
    \sigma_{k,9+k} = \gamma_{k,0}\sqrt{V_0}\sqrt{1-\rho_{k,0}^2}, \qquad
    \sigma_{k,14} = \gamma_{k,0}\sqrt{V_0}\rho_{k,0},
\end{aligned}
\end{equation*}
with all other entries being zero.

In this way, the matrix $\sigma(t)$ provides a complete characterization of both idiosyncratic and systematic sources of risk, and the wealth equation \eqref{dynamics} offers a tractable representation of portfolio dynamics under the proposed multifactor stochastic volatility framework.

The above construction is presented for four risky assets in order to match the
empirical implementation used in the numerical experiments. The same modeling
principle can be extended in parallel to a larger cross-section of assets. In
general, one may introduce one common market factor and asset-specific volatility factors for each risky asset, together with common and idiosyncratic Brownian drivers. The resulting volatility matrix is obtained by adding the corresponding rows and columns in the same way as in the four-asset specification. Hence the
framework is not intrinsically restricted to four assets; the current dimension is chosen for empirical tractability and numerical demonstration.

\begin{remark}
The CIR-type variance factors in \eqref{CIR} are not bounded in general, and therefore the induced drift and volatility coefficients need not be
bounded stochastic processes. A fully rigorous well-posedness theory for the
corresponding SREs under this unbounded empirical specification would require
additional mathematical arguments and is beyond the scope of the present paper. The objective of this research is instead to provide a realistic and tractable empirical specification for numerical portfolio construction. In the Monte Carlo implementation and in the historical backtests, the simulated or realized factor paths remain finite over the finite testing horizon, and the neural-network training is stable in practice. Thus, while the unbounded-factor specification requires further theoretical justification at the level of SRE well-posedness, it is suitable for the numerical and empirical analysis carried out in this work.
\end{remark}

\subsection{Historical Data and Simulation}

To calibrate the proposed multifactor stochastic volatility model, we construct
a daily market dataset consisting of the S\&P 500 Index and four sector SPDR
ETFs: Technology (XLK), Financials (XLF), Energy (XLE), and Health Care (XLV). These sectors provide a diversified representation of growth, cyclical, energy, and defensive components of the equity market.

We use the CBOE Volatility Index (VIX) as a proxy for the common market variance
factor. The VIX reflects the risk-neutral expectation of 30-day forward-looking volatility implied by S\&P 500 options, and therefore captures forward-looking market uncertainty and investor risk sentiment.

The in-sample calibration period runs from January 2015 to December 2019.
Figure~\ref{fig:empirical_data} reports the normalized cumulative price
trajectories of the four sector ETFs and the S\&P 500 Index, together with the
evolution of the VIX over the same period.

\begin{figure}[htbp]
    \centering
    \includegraphics[width=0.8\linewidth]{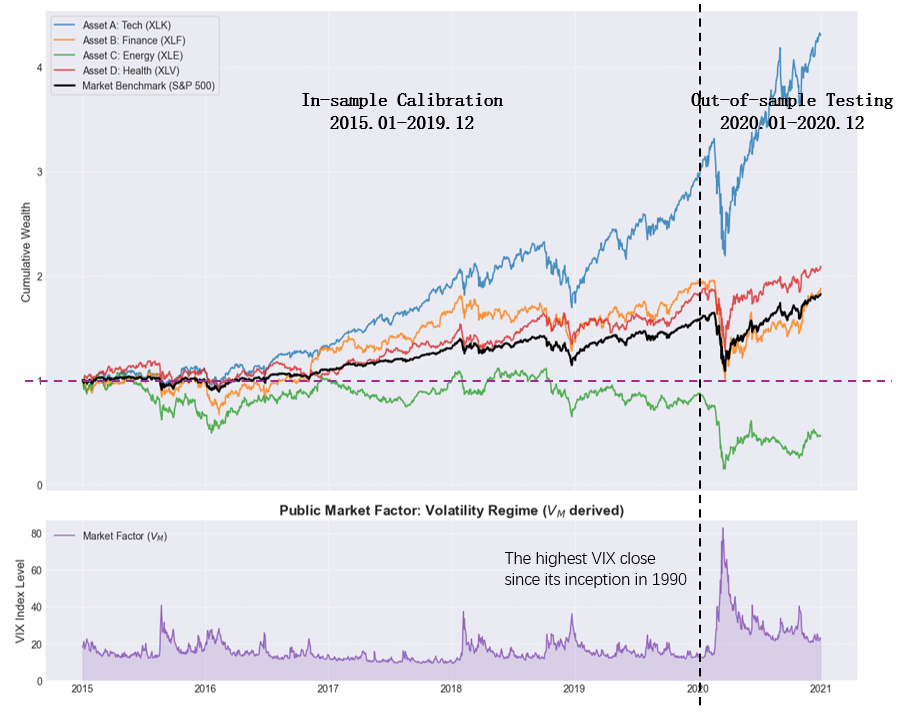}
    \caption{Normalized cumulative prices of the ETFs and the S\&P 500 Index,
    together with the VIX.}
    \label{fig:empirical_data}
\end{figure}

After calibration, we simulate the proposed multifactor model by Monte Carlo. The CIR variance factors are generated using the exact transition law of the square-root process. Specifically, the transition distribution is sampled from a
noncentral chi-square distribution, implemented through a Poisson--Gamma mixture scheme. This avoids the discretization bias of a naive Euler scheme and preserves the one-step distributional moments of the CIR factors.

To verify the consistency of the simulation, we compare the sample statistics of the simulated paths with the corresponding empirical calibration targets. Table~\ref{tab:verification} reports the verification results based on two sets of moments: annualized mean returns for the market index and individual assets, and second-order statistics including volatilities and the cross-asset correlation matrix.

\begin{table}[htbp]
    \centering
    \caption{Monte Carlo simulation verification. The table compares historical
    statistics from January 2015 to December 2019 with ensemble averages from
    \(1{,}000\) simulated scenarios.}
    \label{tab:verification}
    \includegraphics[width=1.0\linewidth]{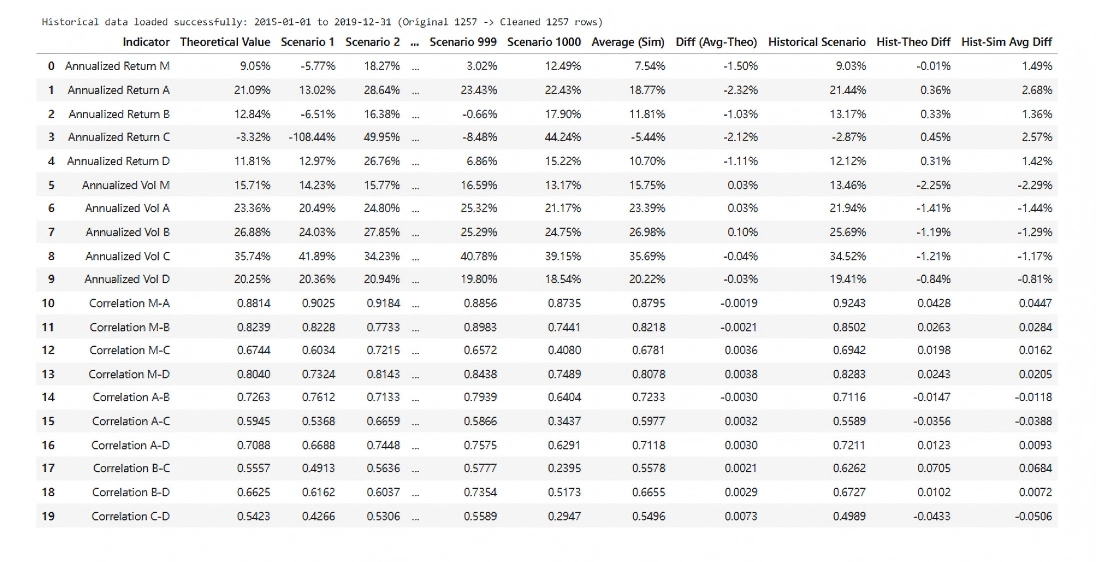}
\end{table}

\section{Numerical Results of SRE}

\subsection{Deep BSDE Training Framework}

With the model specifications and historically calibrated parameters in place, we now turn to the numerical experiments, focusing on both the training dynamics and the final convergence performance.

For the linear BSDEs \eqref{linear-P1}, \eqref{linear-K1} and the logarithmically transformed SRE \eqref{log-transformation}, we first employ the deep BSDE framework of Han et al. \cite{han2018solving}, in which the initial value of the backward process is treated as a trainable parameter and the martingale component is approximated by neural networks. As discussed above, the bounds obtained from the iterative scheme provide useful initial estimates for this trainable parameter. For the logarithmically transformed SRE \eqref{log-transformation}, we additionally implement the DBDP2 algorithm proposed by Huré et al. \cite{hure2020deep}. In the following experiments, we show these methods in terms of training stability, computational efficiency, and approximation accuracy.

\subsection{Numerical Results for the Linear BSDEs}

We first evaluate the transformed linear BSDEs \eqref{linear-P1} and \eqref{linear-K1}, which provide lower and
upper bound estimates for the stochastic Riccati equation \eqref{sre}. The numerical results at \(t=0\) are reported in Table~\ref {tab:linear_initial_values}. The values of \(\widetilde P_1\) and \(\widetilde K_1\) are obtained by solving the
corresponding linear BSDEs using the Deep BSDE method. The columns ``DeepBSDE'' and ``DBDP'' report the initial values of the logarithmic SRE \eqref{log-transformation} computed by the two solvers.

\begin{table}[htbp]
    \centering
    \caption{Initial-value comparison for the linear BSDE bounds and nonlinear SRE solvers.}
    \label{tab:linear_initial_values}
    \renewcommand{\arraystretch}{1.15}
    \begin{tabular}{cccc}
        \toprule
        \(P_1\) \((\widetilde P_1)\) & DeepBSDE & DBDP2 & \(\widetilde K_1\) \\
        \midrule
        \(0.259213\) \((3.857834)\) & \(0.280040\) & \(0.279717\) & \(0.282225\) \\
        \bottomrule
    \end{tabular}
\end{table}

The linear BSDE estimates provide a useful numerical reference for the nonlinear logarithmic SRE solvers. In particular, the learned initial values from DeepBSDE and DBDP are close to each other and lie within the range indicated by the linear comparison bounds. Detailed training histories and out-of-sample path distributions for the linear BSDEs are provided in
Appendix~\ref{app:linear_bsde_figures}.

To further assess the accuracy of the linear BSDE approximation, we conduct an
out-of-sample evaluation over \(50{,}000\) independent Monte Carlo paths.
Table~\ref{tab:oos_statistics} reports the terminal statistics for
\(\widetilde K_1\) and \(\widetilde P_1\). Both terminal means are close to the target value \(1.0\), with small terminal MSEs and moderate tail dispersion, indicating that the learned solutions satisfy the terminal condition accurately under independent testing.

\begin{table}[htbp]
    \centering
    \caption{Out-of-sample terminal statistics for the linear BSDEs.}
    \label{tab:oos_statistics}
    \renewcommand{\arraystretch}{1.15}
    \begin{tabular}{ccc}
        \toprule
        \textbf{Metric} & \(\widetilde K_1\) & \(\widetilde P_1\) \\
        \midrule
        Terminal mean target \(1.0\) & \(0.999878\) & \(0.999317\) \\
        Terminal standard deviation (Std) & \(0.020500\) & \(0.019084\) \\
        Terminal mean squared error (MSE) & \(4.20\times 10^{-4}\) & \(3.65\times 10^{-4}\) \\
        \(99\%\) percentile & \(1.0546\) & \(1.0361\) \\
        \(1\%\) percentile & \(0.9571\) & \(0.9389\) \\
        \bottomrule
    \end{tabular}
\end{table}

\subsection{Numerical Results for the Logarithmic SRE}

Although the original SRE \eqref{sre} can also be approximated directly, its training is slower and less stable in our experiments, mainly due to the positivity
constraint and the reciprocal terms in the Riccati structure. Therefore, we
focus on the logarithmically transformed SRE \eqref{log-transformation} in the reported numerical results.
The corresponding out-of-sample path-distribution diagnostics are deferred to Appendix~\ref{app:log_bsde_figures}.

Table~\ref{tab:oos_statistics_comparison} reports the out-of-sample terminal
statistics of the transformed SRE \eqref{sre}. The log-space columns evaluate the terminal error relative to the target \(\check P_T=0\), while the SRE columns report the corresponding quantities after mapping back to the original variable.

\begin{table}[htbp]
    \centering
    \caption{Out-of-sample terminal statistics for the logarithmic BSDE.}
    \label{tab:oos_statistics_comparison}
    \renewcommand{\arraystretch}{1.15}
    \begin{tabular}{ccccc}
        \toprule
        & \multicolumn{2}{c}{\textbf{Deep BSDE}} 
        & \multicolumn{2}{c}{\textbf{DBDP2}} \\
        \cmidrule(lr){2-3} \cmidrule(lr){4-5}
        \textbf{Metric} 
        & \textbf{Log-space} & \textbf{SRE} 
        & \textbf{Log-space} & \textbf{SRE} \\
        \midrule
        Initial value 
        & \(\log P(0)=-1.27282\) & \(P(0)=0.28004\)
        & \(\log P(0)=-1.27397\) & \(P(0)=0.27971\) \\
        Terminal mean 
        & \(-0.000008\) & \(0.996308\)
        & \(0.000314\) & \(0.997090\) \\
        Std 
        & \(0.009884\) & \(0.010325\)
        & \(0.032782\) & \(0.033544\) \\
        Terminal MSE 
        & \(9.77\times10^{-5}\) & \(1.20\times10^{-4}\)
        & \(1.07\times10^{-3}\) & \(1.16\times10^{-3}\) \\
        \(99\%\) percentile 
        & \(0.0250\) & \(1.0224\)
        & \(0.1063\) & \(1.1067\) \\
        \(1\%\) percentile 
        & \(-0.0225\) & \(0.9721\)
        & \(-0.0519\) & \(0.9442\) \\
        \bottomrule
    \end{tabular}
\end{table}

The terminal statistics show that both neural solvers are consistent with the
zero terminal condition in the logarithmic space. In the Deep BSDE implementation, we initialize the BSDE initial value using a relatively accurate estimate obtained from the linear comparison BSDEs. This favorable initialization helps the Deep BSDE training converge faster and achieve higher terminal accuracy in the present experiment. However, this observation should not be interpreted as
evidence that Deep BSDE is generally superior to DBDP2. The current example is
low-dimensional and is intended to validate the logarithmic formulation, rather than to provide a systematic benchmark comparison between the two algorithms.

\section{Market Investment Strategy}

We finally use the numerical solution of the SRE \eqref{sre} to compute the efficient frontier and investment strategy \eqref{strategy} of the mean--variance portfolio
problem as shown in Appendix \ref{solving-MV}.

\subsection{Numerical Efficient Frontier}

The numerical efficient frontier is plotted in the variance-mean plane as shown in Figure \ref{fig:efficient_frontier}, where the horizontal axis represents the terminal wealth variance $\mathcal{J}^*(d) = \operatorname{Var}[X(T)]$ and the vertical axis denotes the expected terminal wealth $d = \mathbb{E}[X(T)]$. The efficient frontier forms a parabolic curve, explicitly illustrating the quadratic trade-off between the expected return target and the portfolio risk. The global minimum variance portfolio is located at the vertical intercept $\mathcal{J}^*(d) = 0$, where the strategy entirely eliminates risk, and the expected terminal wealth precisely equals the deterministic risk-free compounding trajectory, given by $d = x_0 e^{rT}$.

\begin{figure}[htbp]
    \centering
    \includegraphics[width=0.6\textwidth]{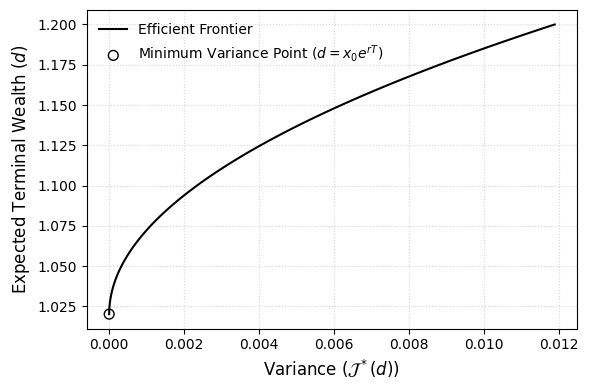}
    \caption{The mean-variance efficient frontier plotted in the variance-mean plane.}
    \label{fig:efficient_frontier}
\end{figure}

\subsection{Out-of-Sample Backtest Portfolio Performance}

We further evaluate the portfolio policies induced by the numerically computed SREs in an out-of-sample backtest investment experiment. The purpose of this experiment is to examine whether the stochastic control policy obtained from the mean--variance formulation can produce stable and valid wealth dynamics under realistic market fluctuations.

We compare the proposed mean--variance strategies with three standard empirical benchmarks and one simplified continuous-time baseline. The empirical benchmarks include the equal-weight strategy (EW), the inverse-variance strategy (IV), and the global minimum-variance strategy (GMV). The EW strategy allocates capital
uniformly across all risky assets. The IV strategy assigns weights inversely
proportional to the historical variances of individual assets, ignoring
cross-asset correlations. The GMV strategy uses the full historical covariance matrix \(\Sigma\) and is given by
\[
    w_{\rm GMV}
    =
    \frac{\Sigma^{-1}\mathbf 1}
    {\mathbf 1^* \Sigma^{-1}\mathbf 1}.
\]
In addition, we include a Black--Scholes baseline with constant drift and
volatility parameters calibrated from a rolling historical window. Under this
simplified model, the state-dependent stochastic-volatility structure is
removed, and the resulting strategy serves as a reference for assessing the
value of the multifactor formulation.

We conduct out-of-sample backtests over two testing periods, 2020 and 2025. The
2020 sample includes the pandemic-induced market stress period, while the 2025
sample is used to assess performance in a different market environment. In both
cases, the initial wealth is \(x_0=100\), the annualized target return is
\(6\%\), and the portfolio is rebalanced every five trading days. Based on the
daily net asset value (NAV) trajectory, we compute annualized return,
annualized volatility, Sharpe ratio, Sortino ratio, maximum drawdown (MDD),
recovery time (RT), Calmar ratio, and marginal expected shortfall at the
\(5\%\) level, \(\mathrm{MES}_{5\%}\).

We evaluate the resulting portfolio policies in two out-of-sample periods:
the stress market of 2020 and the more recent market environment of 2025.
The year 2020 provides a severe stress test because it contains the
pandemic-induced market crash, while the 2025 experiment is used to examine
whether the portfolio construction remains stable in a different market regime. In both cases, the initial wealth is set to \(x_0=100\), the annualized target return is fixed at \(6\%\), and the portfolio is rebalanced every five trading days.

Tables~\ref{tab:risk_metrics_2020} and~\ref{tab:risk_metrics_2025} report the
risk and performance metrics for the two testing periods.

\begin{table}[htbp]
    \centering
    \caption{Out-of-sample risk and performance metrics for the year 2020.}
    \label{tab:risk_metrics_2020}
    \renewcommand{\arraystretch}{1.15}
    \small
    \resizebox{\textwidth}{!}{
    \begin{tabular}{lrrrrrrrr}
        \toprule
        \textbf{Strategy} 
        & \textbf{Return} 
        & \textbf{Vol} 
        & \textbf{Sharpe} 
        & \textbf{Sortino} 
        & \textbf{Calmar} 
        & \textbf{MDD} 
        & \textbf{RT} 
        & \(\boldsymbol{\mathrm{MES}_{5\%}}\) \\
        \midrule
        MV Deep-BSDE & 5.36\% & 12.40\% & 0.43 & 0.68 & 0.56 & -9.58\% & 83 & -0.39\% \\
        MV DBDP & 6.01\% & 11.43\% & 0.53 & 0.86 & 0.76 & -7.93\% & 52 & -0.39\% \\
        BS Baseline    & 6.87\% & 32.19\% & 0.21 & 0.30 & 0.24 & -28.34\% & 61 & -4.48\% \\
        EW             & 2.99\% & 58.01\% & 0.05 & 0.07 & 0.06 & -52.62\% & 222 & -9.51\% \\
        IV             & 6.95\% & 42.92\% & 0.16 & 0.22 & 0.16 & -42.49\% & 219 & -6.90\% \\
        GMV            & 18.44\% & 18.23\% & 1.01 & 1.44 & 2.16 & -8.55\% & 79 & -1.51\% \\
        \bottomrule
    \end{tabular}
    }
\end{table}

\begin{table}[htbp]
    \centering
    \caption{Out-of-sample risk and performance metrics for the year 2025.}
    \label{tab:risk_metrics_2025}
    \renewcommand{\arraystretch}{1.15}
    \small
    \resizebox{\textwidth}{!}{
    \begin{tabular}{lrrrrrrrr}
        \toprule
        \textbf{Strategy} 
        & \textbf{Return} 
        & \textbf{Vol} 
        & \textbf{Sharpe} 
        & \textbf{Sortino} 
        & \textbf{Calmar} 
        & \textbf{MDD} 
        & \textbf{RT} 
        & \(\boldsymbol{\mathrm{MES}_{5\%}}\) \\
        \midrule
        MV Deep BSDE & \(7.53\%\)  & \(10.40\%\) & \(0.72\) & \(1.19\) & \(0.77\) & \(-9.73\%\)  & \(75\)  & \(-0.92\%\) \\
        MV DBDP      & \(7.42\%\)  & \(11.64\%\) & \(0.64\) & \(1.05\) & \(0.72\) & \(-10.24\%\) & \(69\)  & \(-0.99\%\) \\
        BS Baseline  & \(7.07\%\)  & \(10.16\%\) & \(0.70\) & \(1.27\) & \(1.05\) & \(-6.71\%\)  & \(34\)  & \(-0.97\%\) \\
        EW           & \(16.22\%\) & \(18.14\%\) & \(0.89\) & \(1.24\) & \(0.97\) & \(-16.80\%\) & \(92\)  & \(-2.72\%\) \\
        IV           & \(14.21\%\) & \(16.70\%\) & \(0.85\) & \(1.18\) & \(0.92\) & \(-15.47\%\) & \(153\) & \(-2.43\%\) \\
        GMV          & \(3.51\%\)  & \(12.57\%\) & \(0.28\) & \(0.37\) & \(0.22\) & \(-16.16\%\) & \(226\) & \(-1.60\%\) \\
        \bottomrule
    \end{tabular}
    }
\end{table}

Figures~\ref{fig:strategy_comparison_2020} and~\ref{fig:strategy_comparison_2025}
plot the corresponding cumulative wealth trajectories.

\begin{figure}[htbp]
    \centering
    \includegraphics[width=0.9\linewidth]{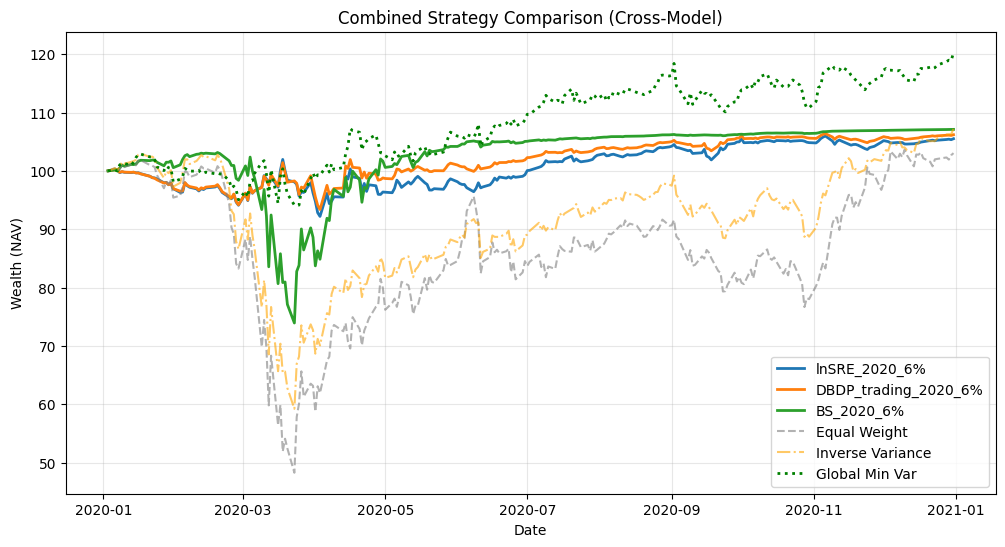}
    \caption{Cumulative wealth trajectories of all strategies in 2020.}
    \label{fig:strategy_comparison_2020}
\end{figure}

\begin{figure}[htbp]
    \centering
    \includegraphics[width=0.9\linewidth]{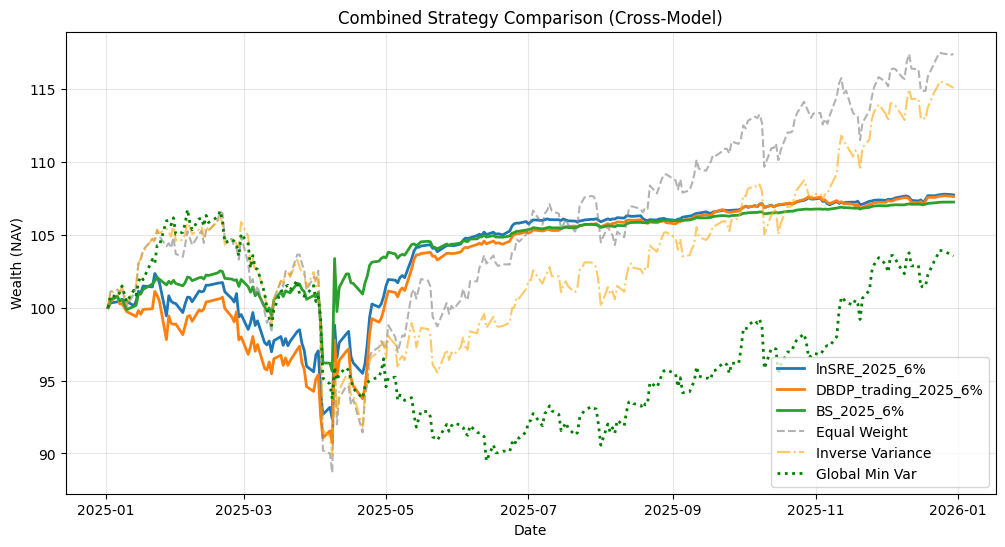}
    \caption{Cumulative wealth trajectories of all strategies in 2025.}
    \label{fig:strategy_comparison_2025}
\end{figure}

The two testing periods highlight complementary aspects of the proposed
mean--variance strategy. In the 2020 stress period, the MV strategies generated
by the SRE solvers achieve realized returns close to the prescribed target while substantially reducing volatility, maximum drawdown, and tail exposure relative to the EW, IV, and Black--Scholes baselines. This is also visible from the wealth curves, where the MV strategies exhibit much smoother dynamics during the market crash. Although the GMV benchmark attains the highest realized return and Sharpe ratio in 2020, but its performance in 2025 was the worst. Compared with EW and IV, the MV strategies sacrifice some upside return in 2025, but they also achieve smoother wealth trajectories, lower volatility, smaller maximum drawdowns, and less negative \(\mathrm{MES}_{5\%}\). 

Compared with the Black--Scholes baseline, the advantage of the proposed
multifactor MV strategy is most evident in 2020. During this stress period, the
multifactor strategy delivers much lower volatility, smaller drawdown, and
milder \(\mathrm{MES}_{5\%}\), indicating stronger robustness under severe
market fluctuations. In 2025, the two mean--variance specifications perform
more similarly, with comparable volatility, Sharpe ratio, and tail-risk
measures. This suggests that the state-dependent multifactor structure is
particularly valuable in stressed markets, while remaining competitive with the simpler constant-coefficient Black--Scholes baseline in a more regular market environment. This comparison is economically relevant because both are mean--variance policies, while the Black--Scholes baseline uses constant coefficients and the proposed model uses state-dependent multifactor dynamics.

Overall, the combined evidence from 2020 and 2025 supports the numerical
usefulness of the reflected multifactor stochastic control framework. The main
advantage is not that the proposed MV strategy dominates every benchmark in
every metric, but that it produces stable target-return portfolio dynamics with
strong drawdown control across different market regimes. The wealth curves and
risk statistics together indicate that the state-dependent multifactor
formulation provides a more adaptive and robust portfolio construction than the
constant-coefficient Black--Scholes baseline, especially when market conditions change substantially.


\section{Concluding Remarks}

This paper investigates a computable and empirically implementable framework for
continuous-time mean--variance optimal portfolio selection with random market
coefficients. The main advantage of the proposed model is that it captures state-dependent risk premia, stochastic volatility, and dynamic cross-asset dependence through a tractable multifactor structure. Numerically, we first use two linear BSDEs to obtain a reliable estimate of the initial value, and then solve the logarithmically transformed SRE using both the Deep BSDE and DBDP methods. Empirically, the proposed multifactor mean--variance strategy demonstrates its
main advantage in producing smoother target-return wealth dynamics with stronger drawdown control and more robust downside-risk protection than the benchmark strategies.

Future work may extend the framework to optimal investment problems under other
utility functions, such as CARA and CRRA preferences. Another direction is to
apply the BSDE-based numerical method to constrained stochastic
control problems, including portfolio, leverage, and trading constraints.

\section*{Supplementary material}

These supplementary materials include the code for the multifactor stochastic volatility model, the implementation of the Deep BSDE and DBDP2
solvers, additional portfolio-allocation figures, and further financial
interpretations of the empirical results.

The detailed parameter values, the historical dataset, BSDE solver implementation, and additional numerical outputs are available on the accompanying open-source project webpage:
\begin{equation*}
\text{\url{https://doi.org/10.5281/zenodo.21838225}}.
\end{equation*}

\section*{Acknowledgments}

The authors would like to thank Weiran Xiong for helpful discussions on the use of deep learning methods for solving backward stochastic differential equations (BSDEs).

\appendix
\setcounter{theorem}{0}
\renewcommand{\thetheorem}{\thesection.\arabic{theorem}}
\setcounter{lemma}{0}
\renewcommand{\thelemma}{\thesection.\arabic{lemma}}
\setcounter{assumption}{0}
\renewcommand{\theassumption}{\thesection.\arabic{assumption}}

\section{Solving the MV Portfolio Problem \ref{prob1}} \label{solving-MV}

In this appendix, we will briefly provide the solution of the classic MV problem in a incomplete market under pre-commitment strategy. Since {\bf Problem \ref{prob1}} is a convex optimization problem with a linear equality constraint, we apply the Lagrange multiplier method. By introducing a multiplier $\lambda \in \mathbb{R}$ to handle the terminal expectation constraint $\mathbb{E}[X(T)] = d$, the constrained {\bf Problem \ref{prob1}} is transformed into an unconstrained stochastic control problem. Specifically, for each fixed $\lambda \in \mathbb{R}$, we consider:
\begin{equation} \label{prob2}
    \begin{array}{rl}
        \min\limits_{\pi(\cdot) \in \mathcal{A}} & J(\pi(\cdot), d, \lambda) \triangleq \mathbb{E}\left[|X(T) - d|^2\right] + 2\lambda(\mathbb{E}[X(T)] - d), \\
        \text{subject to} & (X(\cdot), \pi(\cdot)) \text{ satisfies } \eqref{dynamics}.
    \end{array}
\end{equation}

By completing the square, it is straightforward to see that minimizing $J(\pi(\cdot), d, \lambda)$ is equivalent to solving the following standard stochastic linear-quadratic (LQ) optimal control problem:
\begin{equation} \label{prob3}
    \begin{array}{rl}
        \min\limits_{\pi(\cdot) \in \mathcal{A}} & \mathscr{J}(\pi(\cdot), b) \triangleq \mathbb{E}\left[|X(T) - b|^2\right], \\
        \text{subject to} & \left\{
        \begin{array}{l}
            b = d - \lambda, \\
            (X(\cdot), \pi(\cdot)) \text{ satisfies } \eqref{dynamics}.
        \end{array}
        \right.
    \end{array}
\end{equation}

Before addressing the constrained {\bf Problem \ref{prob1}}, we introduce a feasibility condition to guarantee that the set of admissible strategies satisfying the terminal expectation constraint $\mathbb{E}[X(T)] = d$ is non-empty.

\begin{lemma} \label{feasible-condition}
The constrained stochastic LQ {\bf Problem \ref{prob1}} is feasible for every $d \in \mathbb{R}$ if and only if
\begin{equation} \label{posi-definite}
    \mathbb{E} \left[ \int_{0}^{T} | \tilde{\mu}(t) y(t) + \sigma(t)z(t) |^2 \dt \right] > 0,
\end{equation}
where the pair $(y(\cdot), z(\cdot)) \in L_{\mathbb{F}}^{\infty}(0, T ; \mathbb{R}) \times L_{\mathbb{F}}^{2}(0, T ; \mathbb{R}^{n})$ is the unique solution to the following linear backward stochastic differential equation (BSDE):
\begin{equation} \label{auxiliary}
    \left\{\begin{aligned}
        \mathrm{d}y(t) &= -r(t)y(t) \dt + \langle z(t), \dw(t) \rangle, \qquad t \in [0, T), \\
        y(T) &= 1.
    \end{aligned}\right.
\end{equation}
\end{lemma}

The proof of Lemma \ref{feasible-condition} relies on arguments similar to those in, e.g., Proposition 3.4 of \cite{lim2002mean}, and is therefore omitted. We remark that condition \eqref{posi-definite} is quite mild, even when all market parameters are random. Consequently, we assume throughout this paper that this feasibility condition holds.

Since {\bf Problems \ref{prob1}}, {\bf \ref{prob2}} and {\bf \ref{prob3}} are strictly convex, an optimal strategy $\pi^*(\cdot)$ exists. We denote their corresponding optimal values by $\mathcal{J}^*(d)$, $J^*(d, \lambda)$, and $\mathscr{J}^*(b)$, respectively.

To solve the unconstrained stochastic LQ {\bf Problem \ref{prob3}}, we introduce the following linear BSDE that features unbounded random coefficients:
\begin{equation} \label{BSDE}
    \left\{\begin{aligned}
        \mathrm{d}H &= \left[r H + \left\langle \rho - \frac{1}{P}  (I_{n \times n} - \sigma^* (\sigma\sigma^*)^{-1} \sigma)  \Lambda, \Gamma \right\rangle\right] \dt + \langle \Gamma,\dw\rangle, \qquad t \in [0, T), \\
         H(T) & = 1.
    \end{aligned} \right.
\end{equation}

There is an extensive body of academic literature studying SRE \eqref{sre} and BSDE \eqref{BSDE}. The well-posedness of these equations is given in the following theorem:

\begin{theorem} \label{sre-positive}
    The SRE \eqref{sre} admits a unique adapted solution
    \begin{equation*}
        (P(\cdot), \Lambda(\cdot)) \in L_{\mathbb{F}}^{\infty,c}(0, T; \mathbb{R}_{\gg 1}) \times L_{\mathbb{F}}^{2}(0, T; \mathbb{R}^n),
    \end{equation*}
    and the linear BSDE \eqref{BSDE} admits a unique solution $(H(\cdot), \Gamma(\cdot))$ such that
        \begin{equation*}
            (H(\cdot), \Gamma(\cdot)) \in L_{\mathbb{F}}^{2,c}\left(0, T ; \mathbb{R}\right) \times L_{\mathbb{F}}^{2}\left(0, T ; \mathbb{R}^n\right).
        \end{equation*}
\end{theorem}

For the proof of the above theorem, we refer the reader to works such as Kohlmann and Tang \cite{kohlmann2002global,kohlmann2003minimization}.

\begin{theorem} \label{slq}
    The stochastic LQ control {\bf Problem \ref{prob3}} is well-posed, with the unique optimal feedback control given by
    \begin{equation*}
        \pi^* = - (\sigma\sigma^*)^{-1} \left[ \left( \tilde{\mu} + \frac{\sigma \Lambda}{P}\right) \left( X + (\lambda-d)H \right)+ (\lambda-d)\sigma \Gamma \right],
    \end{equation*}
    and its corresponding optimal cost functional given by
    \begin{equation*} \label{cost}
        \mathscr{J}^*(d,\lambda) = P(0)|x_0 + (\lambda-d)H(0)|^2 + c_0 (\lambda-d)^2,
    \end{equation*}
    where
    \begin{equation*}
        c_0 :=  \mathbb{E} \left[ \int_{0}^{T}P\Gamma^* [I_{n \times n} - \sigma^* (\sigma\sigma^*)^{-1} \sigma]\Gamma \dt \right].
    \end{equation*}
\end{theorem}

The core idea of the proof relies on applying It\^o's formula to the process $P (X +(\lambda-d) H)^2$. Detailed arguments can be found in, e.g., Kohlmann and Tang \cite{kohlmann2002global}, and thus the proof is omitted here. We can now express the optimal value $J^*(d,\lambda)$ as follows:
\begin{equation*}
    J^*(d,\lambda) = P(0)|x_0 + (\lambda-d)H(0)|^2 + c_0 (\lambda-d)^2 - \lambda^2.
\end{equation*}

\begin{theorem}
    The following strict inequality holds:
    \begin{equation*}
        P(0)H^2(0) + c_0 < 1.
    \end{equation*}
    Moreover, the optimal value of the constrained {\bf Problem \ref{prob1}} is given by
    \begin{equation*}
        \mathcal{J}^*(d) = c_0 d^2 + P(0)(x_0 - H(0)d)^2 + \frac{\left[P(0)H(0)(x_0 - H(0)d) - c_0d\right]^2}{1 - c_0 - P(0)H(0)^2},
    \end{equation*}
    the optimal Lagrange multiplier $\lambda^*(d)$ is given by
    \begin{equation*}
        \lambda^*(d) = \frac{P(0)H(0)\left(x_0-H(0)d-c_0d\right)}{1-c_0-P(0)H(0)^2}.
    \end{equation*}
\end{theorem}

\begin{proof}
    According to the classical Lagrange duality theorem (see, e.g., Luenberger \cite{luenberger1997optimization}), for any target $d \in \mathbb{R}$, the optimal value of the constrained problem satisfies the following duality relationship:
    \begin{equation} \label{Lagrange}
        \mathcal{J}^*(d) = \max_{\lambda \in \mathbb{R}} J^*(d, \lambda).
    \end{equation}
    Since $J^*(d, \lambda)$ is a quadratic function of $\lambda$, its first and second derivatives with respect to $\lambda$ are given by
    \begin{equation*}
        \frac{\partial}{\partial \lambda} J^*(d, \lambda) = 2P(0)H(0) [x_0 + (\lambda - d)H(0)] + 2c_0(\lambda - d) - 2\lambda,
    \end{equation*}
    and
    \begin{equation*}
        \frac{\partial^2}{\partial \lambda^2} J^*(d, \lambda) = 2P(0)H(0)^2 + 2c_0 - 2.
    \end{equation*}
    For the supremum in \eqref{Lagrange} to be finite and attainable for every $d \in \mathbb{R}$, the quadratic function $J^*(d, \lambda)$ must be strictly concave in $\lambda$. This strict concavity implies that the second derivative must be strictly negative, i.e.,
    \begin{equation*}
        2P(0)H(0)^2 + 2c_0 - 2 < 0,
    \end{equation*}
    which immediately yields $P(0)H(0)^2 + c_0 < 1$. 

    Consequently, the optimal dual variable $\lambda^*$ is uniquely determined by the first-order condition $\left. \frac{\partial}{\partial \lambda} J^*(d, \lambda) \right|_{\lambda = \lambda^*} = 0$. Solving this equation for $\lambda^*$ and substituting it back into the expression for $J^*(d, \lambda^*)$ yields the desired formula for $\mathcal{J}^*(d)$.
\end{proof}

To explicitly construct the mean-variance efficient frontier, it is necessary to determine the initial values $P(0)$ and $H(0)$, which depend on the solutions to BSDEs \eqref{sre} and \eqref{BSDE}. In standard financial market models, the risk-free interest rate $r(\cdot)$ is typically assumed to be deterministic, and our multi-factor stochastic volatility model in Section \eqref{empirical} admits this. Motivated by this, we impose the following assumption in our analysis.

\begin{assumption}[Deterministic Interest Rate] \label{assump-r}
    The interest rate $r(\cdot)$ of the risk-free bond is a deterministic function.
\end{assumption}

Under Assumption \ref{assump-r}, by the uniqueness of the solution, the linear BSDE \eqref{BSDE} degenerates into the following ordinary differential equation (ODE):
\begin{equation} \label{ode}
    \left\{\begin{aligned}
        \mathrm{d}H(t) &= r(t) H(t) \dt, \qquad t \in [0, T), \\
        H(T) &= 1.
    \end{aligned}\right.
\end{equation}
Consequently, we obtain the explicit solution:
\begin{equation*}
    (H(t), \Gamma(t)) = \left( \exp\left\{ -\int_{t}^{T} r(s) \ds \right\}, 0 \right).
\end{equation*}
This immediately implies $c_0 = 0$. The optimal cost $\mathcal{J}^*(d)$ then simplifies to
\begin{equation*}
    \mathcal{J}^*(d) = \frac{P(0)(x_0 - H(0)d)^2}{1 - P(0)H(0)^2}.
\end{equation*}
and the optimal Lagrange multiplier $\lambda^*(d)$ simplies to
    \begin{equation*}
        \lambda^*(d) = \frac{P(0)H(0)\left(x_0-H(0)d\right)}{1-P(0)H(0)^2},
    \end{equation*}
    Replace with the value of $\lambda^*(d)-d$ and notice that $\Gamma=0$, we obtain the optimal feedback control of Problem \ref{prob1}:
    \begin{equation} \label{strategy}
        \pi^* = - (\sigma\sigma^*)^{-1} \left( \tilde{\mu} + \frac{\sigma \Lambda}{P}\right)\left(X + \left(\frac{P(0)H(0)x_0-d}{1-P(0)H(0)^2}\right)H \right).
    \end{equation}

Similarly, returning to the auxiliary BSDE \eqref{auxiliary}, it also reduces to an ODE with the unique solution
\begin{equation*}
    (y(t), z(t)) = \left( \exp\left\{ -\int_{t}^{T} r(s) \ds \right\}, 0 \right).
\end{equation*}
Accordingly, the feasibility condition \eqref{posi-definite} is equivalent to
\begin{equation*}
    \mathbb{E} \left[ \int_{0}^{T} |\tilde{\mu}(t)|^2 \dt \right] > 0.
\end{equation*}
This is a highly natural and mild assumption, as it merely requires that the appreciation rates of the risky stocks deviate from the risk-free bond rate in a mean-square sense.

\section{Comparison Theorem for BSDEs with Quadratic Growth} \label{appendix-quadratic}

In this Appendix, we recall some classical results for BSDEs whose generator has quadratic growth in the martingale term. More precisely, consider the BSDE
\begin{equation} \label{quadratic_BSDE}
    y(t) = \xi + \int_t^T f\left(s, y(s), z(s)\right) \mathrm{d}s - \int_t^T z(s) \mathrm{d}W(s), \qquad t \in [0,T],
\end{equation}
under the following assumptions:
\begin{enumerate}
    \item \(\xi \in L_{\mathcal{F}_T}^{\infty}(\mathbb{R})\);
    \item there exists a constant \(C>0\) such that, for all
    \((y,z)\in \mathbb{R}\times \mathbb{R}^n\),
    \[
        |f(t,\omega,y,z)|
        \leq C\bigl(1+|y|+|z|^2\bigr),
        \qquad \mathrm{d}t\otimes \mathrm{d}\mathbb{P}\text{-a.e.};
    \]
    \item there exists a constant \(C>0\) such that, for all
    \((y_i,z_i)\in \mathbb{R}\times \mathbb{R}^n\), \(i=1,2\),
    \[
        |f(t,\omega,y_1,z_1)-f(t,\omega,y_2,z_2)|
        \leq C\left(
        |y_1-y_2|
        +
        \bigl(1+|z_1|+|z_2|\bigr)|z_1-z_2|
        \right),
        \qquad \mathrm{d}t\otimes \mathrm{d}\mathbb{P}\text{-a.e.}
    \]
\end{enumerate}

\begin{theorem}\label{BSDE-quadratic}
Suppose that, for each \(k \geq 0\), the data \((\xi_k,f_k)\) satisfy the above assumptions. Then the BSDE \eqref{quadratic_BSDE} with terminal condition \(\xi_k\) and generator \(f_k\) admits a unique solution
\[ (y_k(\cdot),z_k(\cdot)) \in L_{\mathbb{F}}^{\infty}(0,T;\mathbb{R}) \times L_{\mathbb{F}}^{2}(0,T;\mathbb{R}^n). \] 
Moreover, if \(\xi_1 \leq \xi_2\) a.s. and $f_1(\cdot,y,z) \leq f_2(\cdot,y,z)$ $\mathrm{d}t \otimes \mathrm{d}\mathbb{P}\text{-a.e.}$, for all \((y,z)\in \mathbb{R}\times \mathbb{R}^n\), then
\[
y_1(t) \leq y_2(t), \qquad 0 \leq t \leq T, \quad \text{a.s.} 
\]
In particular, \(y_1(0) \leq y_2(0)\).
\end{theorem}

Since these results are well known, we omit the proofs and refer the reader to the standard literature on quadratic BSDEs (see e.g. Kobylanski \cite{kobylanski2000backward} or Chapter $7$ in Zhang \cite{zhang2017backward}).

\section{Additional Numerical and Portfolio Diagnostics}

\subsection{Training and Out-of-Sample Diagnostics for the Linear BSDEs}
\label{app:linear_bsde_figures}

Figure~\ref{fig:linear_bsde_results} reports the diagnostic plots for the two
linear BSDEs \eqref{linear-P1} and \eqref{linear-K1}. The top row shows the training dynamics, including the learned
initial value, terminal expectation, and MSE loss. The bottom row shows the
out-of-sample path distributions and verifies that both processes are centered
around the terminal target \(1.0\).

\begin{figure}[htbp]
    \centering
    \begin{subfigure}{0.49\textwidth}
        \centering
        \includegraphics[width=\linewidth]{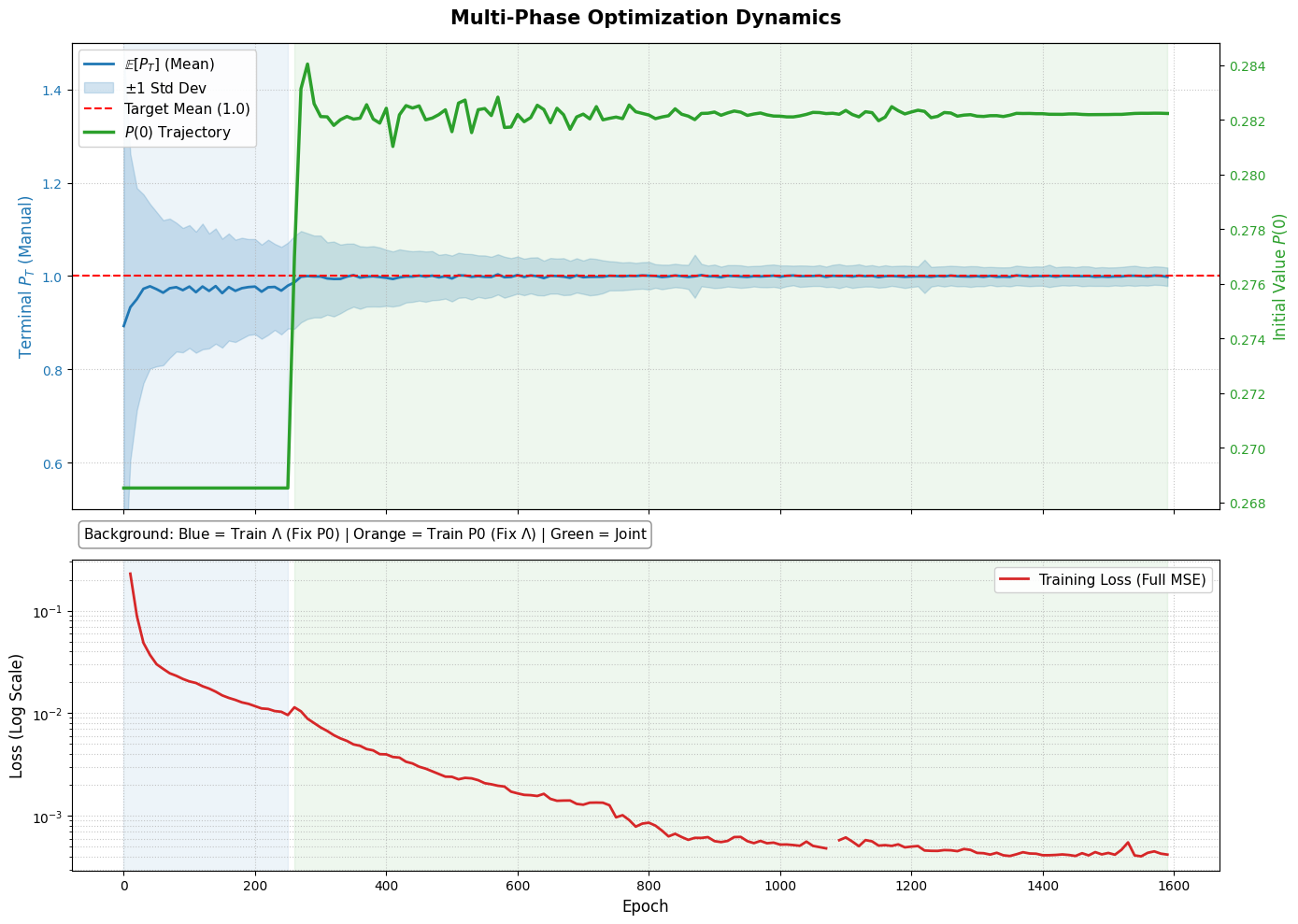}
        \caption{Training: \(\widetilde K_1\)}
        \label{fig:train_tK1}
    \end{subfigure}
    \hfill
    \begin{subfigure}{0.49\textwidth}
        \centering
        \includegraphics[width=\linewidth]{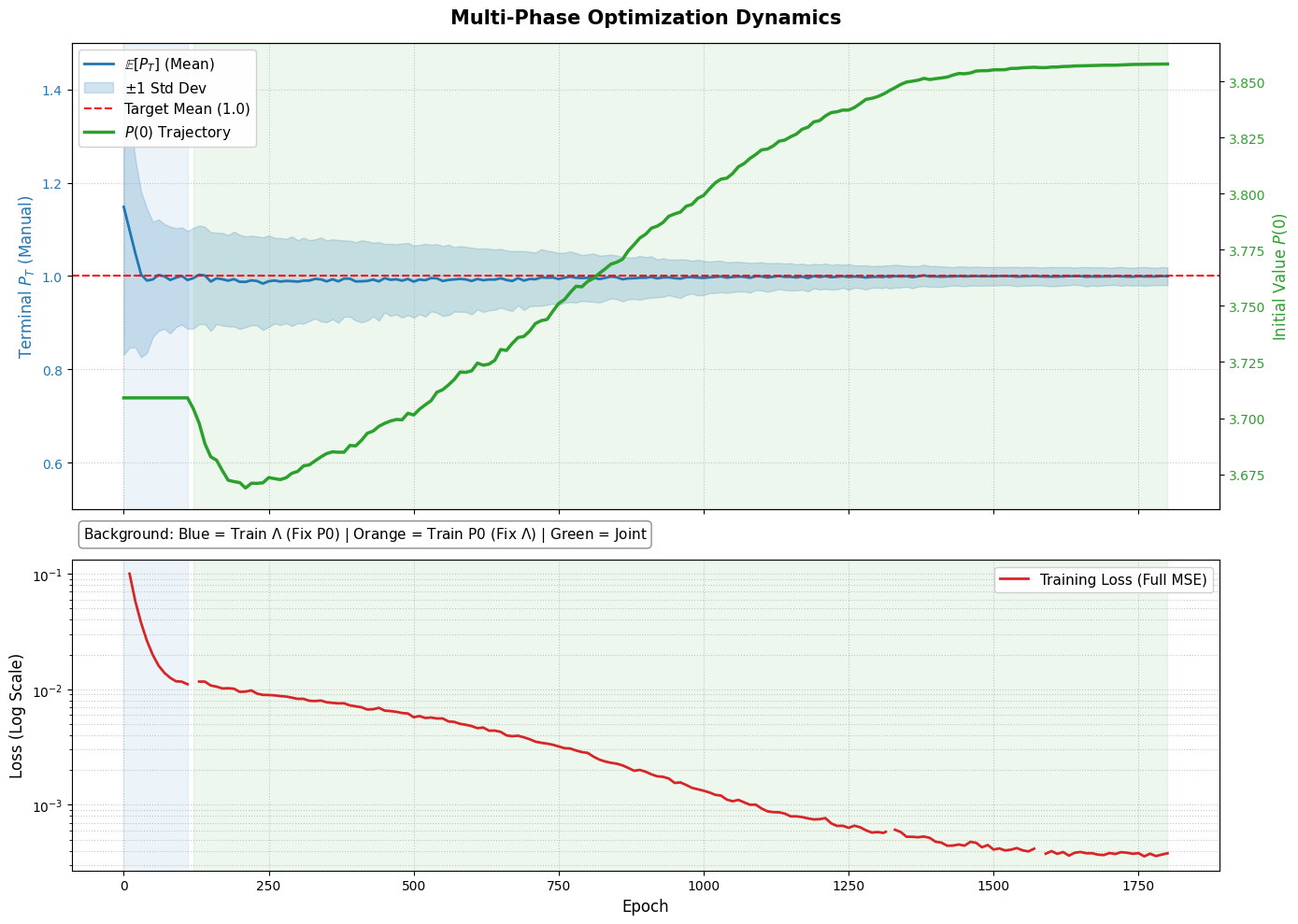}
        \caption{Training: \(\widetilde P_1\)}
        \label{fig:train_tP1}
    \end{subfigure}

    \vspace{0.5cm}

    \begin{subfigure}{0.49\textwidth}
        \centering
        \includegraphics[width=\linewidth]{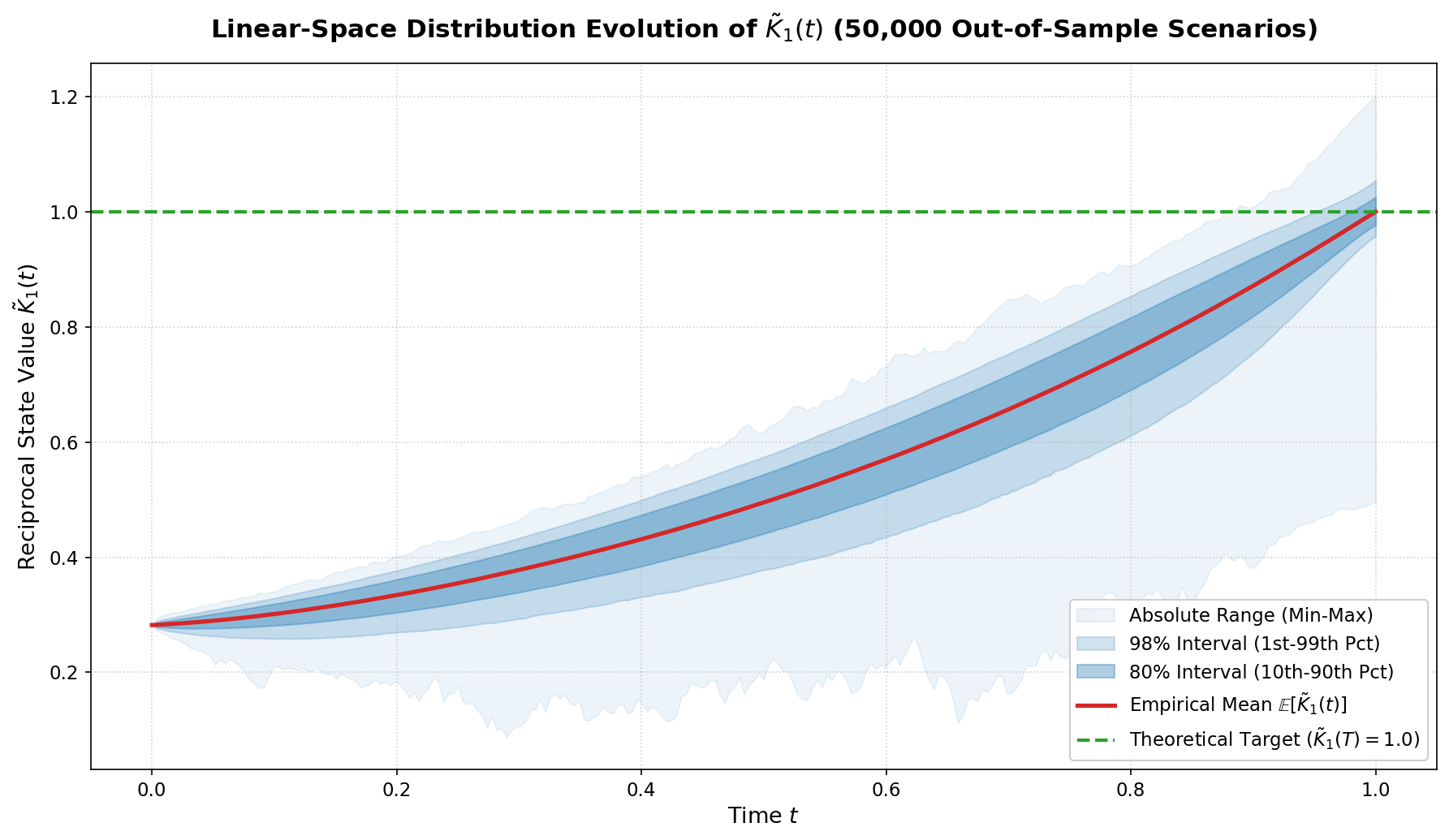}
        \caption{Out-of-sample: \(\widetilde K_1\)}
        \label{fig:oos_tK1}
    \end{subfigure}
    \hfill
    \begin{subfigure}{0.49\textwidth}
        \centering
        \includegraphics[width=\linewidth]{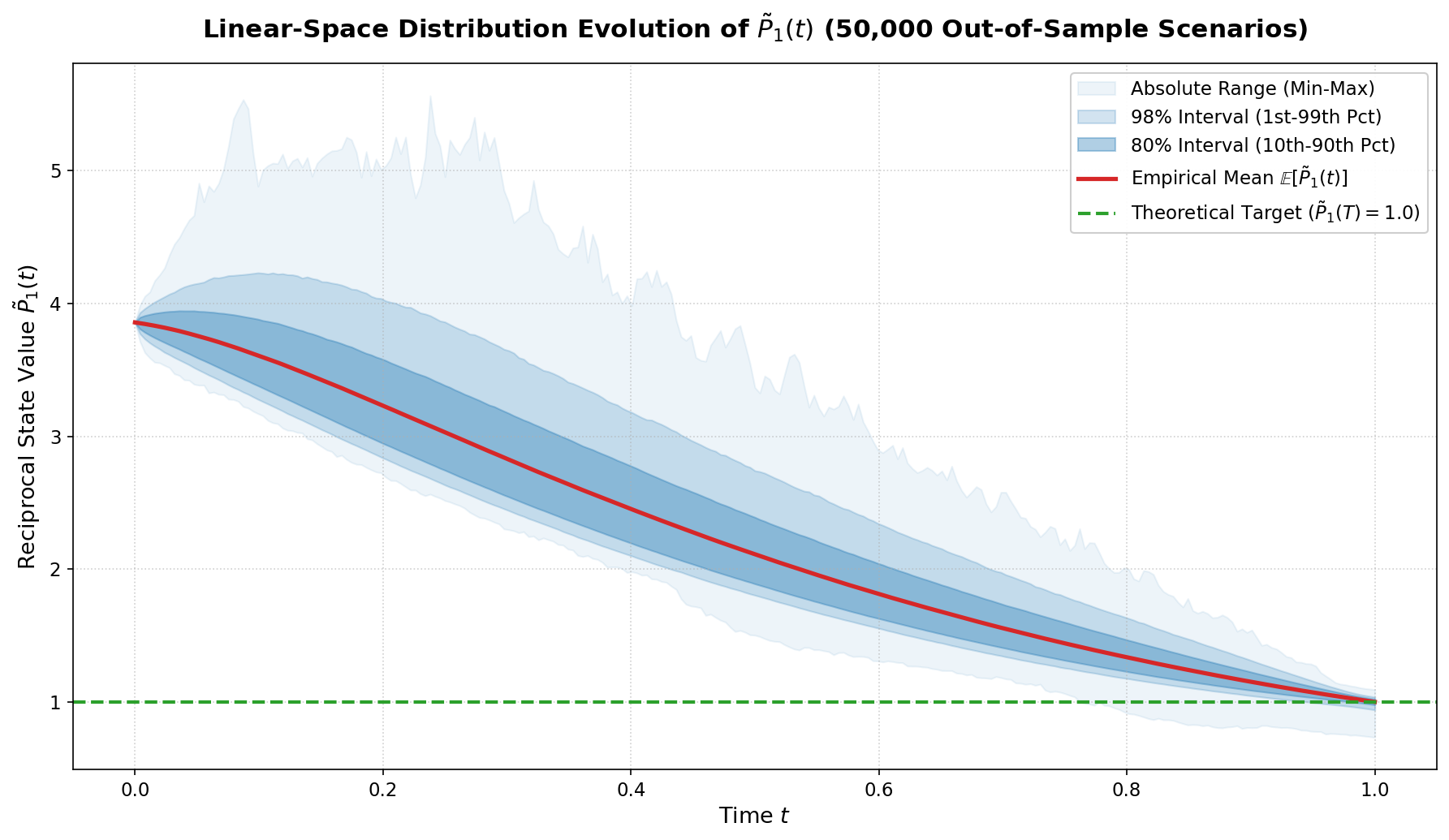}
        \caption{Out-of-sample: \(\widetilde P_1\)}
        \label{fig:oos_tP1}
    \end{subfigure}

    \caption{Training and out-of-sample diagnostics for the linear BSDEs.}
    \label{fig:linear_bsde_results}
\end{figure}

\subsection{Path-Distribution Diagnostics for the Logarithmic BSDE}
\label{app:log_bsde_figures}

Figure~\ref{fig:path_distribution} shows the out-of-sample evolution of the
cross-sectional distribution of \(\check P_t\) under the two neural solvers. The
plots are used as diagnostic illustrations of the transformed dynamics and the
terminal alignment with \(\check P_T=0\).

\begin{figure}[htbp]
    \centering 
    \begin{minipage}{0.48\linewidth}
        \centering
        \includegraphics[width=\linewidth]{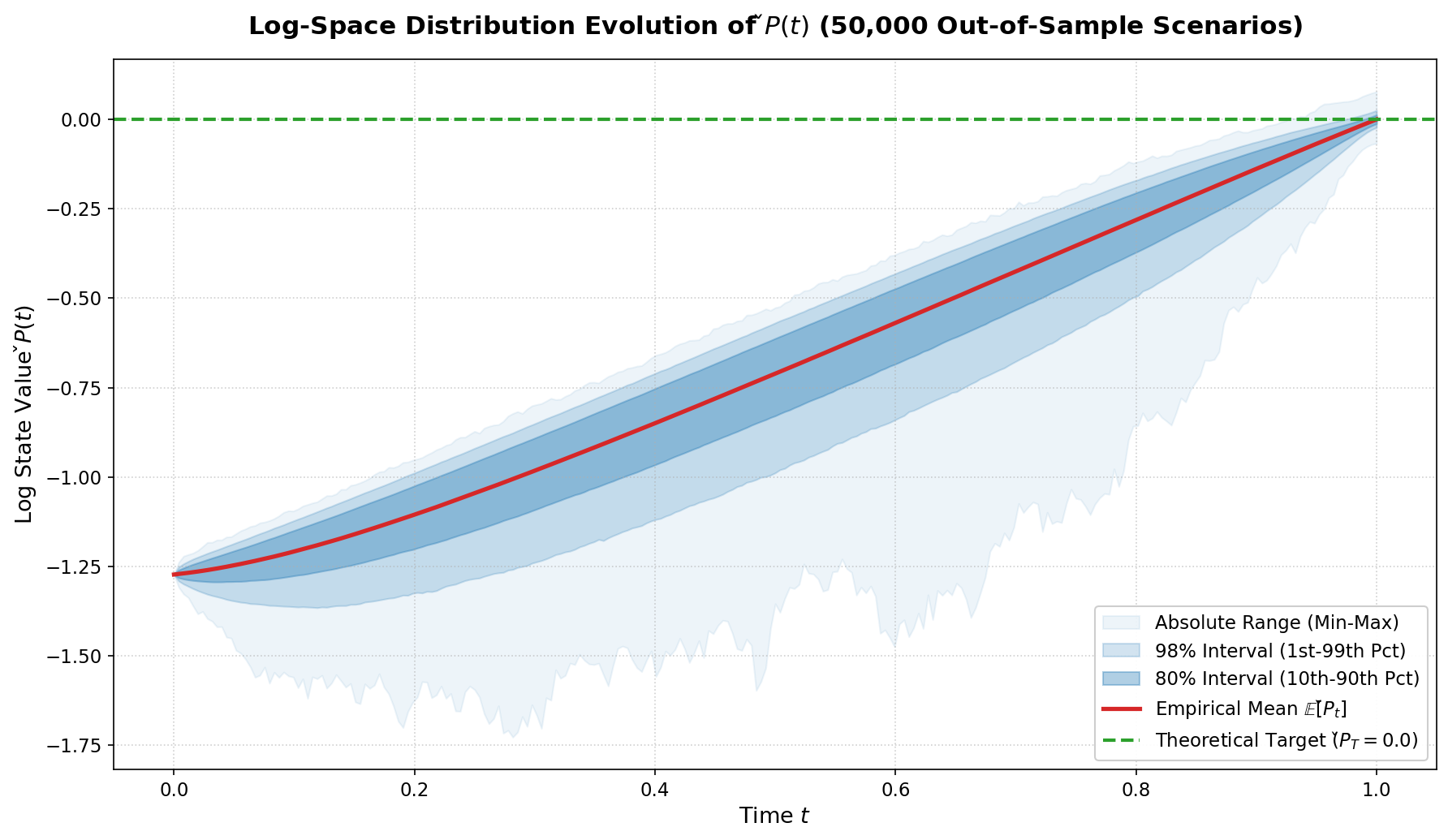}
        \centerline{(a) Deep BSDE method}
    \end{minipage}\hfill
    \begin{minipage}{0.48\linewidth}
        \centering
        \includegraphics[width=\linewidth]{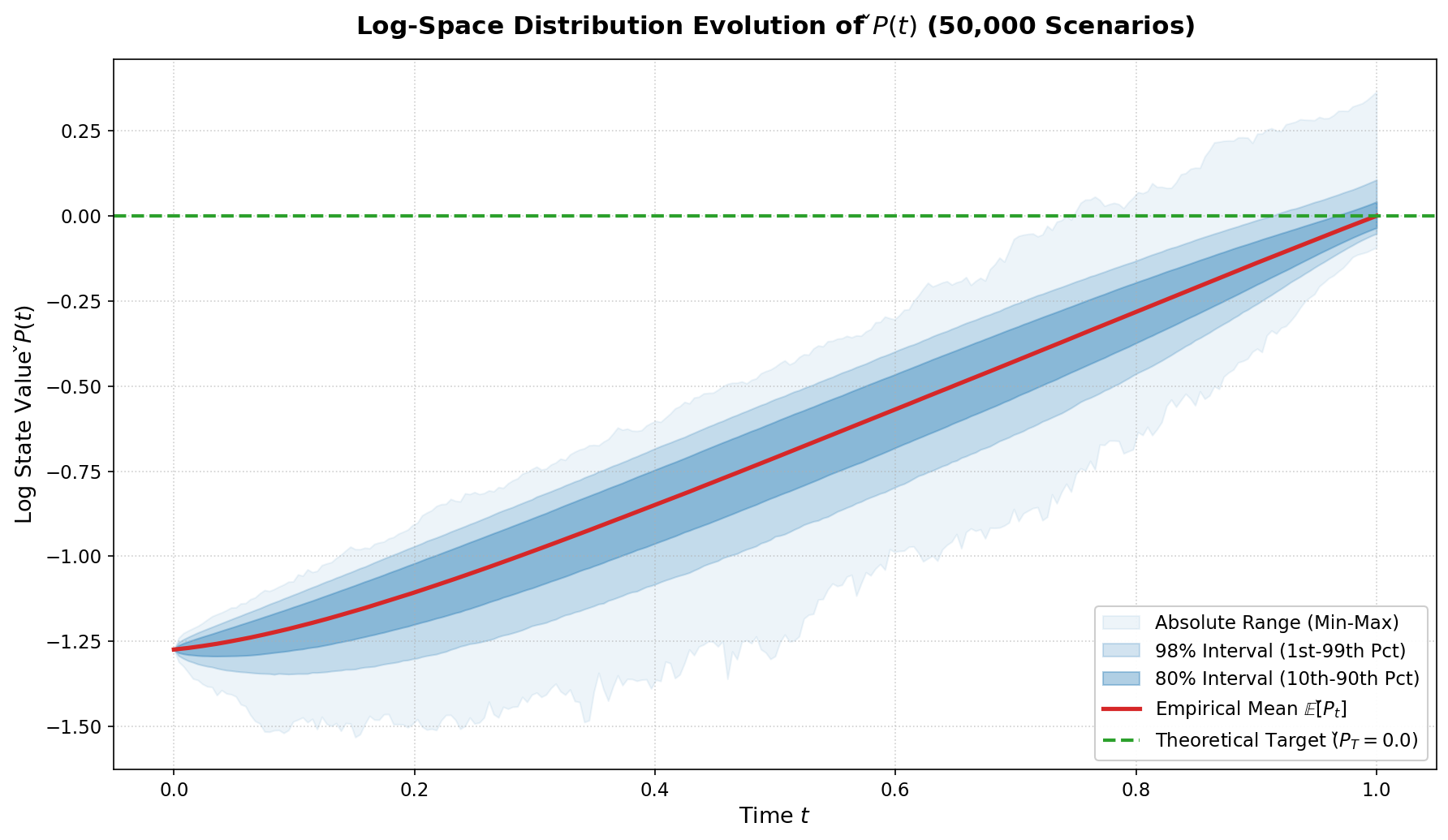} 
        \centerline{(b) DBDP method}
    \end{minipage}
    \caption{Out-of-sample distribution evolution of the transformed process
    \(\check P_t\).}
    \label{fig:path_distribution}
\end{figure}




\newpage
\bibliographystyle{amsplain}

\bibliography{reference}

\end{document}